\documentclass[a4paper,12pt]{amsart}

\makeindex

\usepackage{amsfonts,graphics,amsmath,amsthm,amsfonts,amscd,amssymb,amsmath,latexsym,euscript, enumerate}
\usepackage{epsfig,url}
\usepackage{flafter}
\usepackage[all,cmtip,line]{xy}
\usepackage{array}
\usepackage[english]{babel}
\usepackage{overpic}
\usepackage{subfig}
\usepackage{multirow}

\usepackage{wrapfig}
\usepackage{enumitem}

\usepackage[usenames,dvipsnames,svgnames,table]{xcolor}
\usepackage{graphicx}
\usepackage[bookmarks=true,bookmarksnumbered=true,
 pdftitle={},
 pdfsubject={},
 pdfauthor={},
 pdfkeywords={TeX, dvipdfmx, hyperref, color,},
 colorlinks=true, linkcolor=RoyalBlue, citecolor=Emerald]
 {hyperref}

\numberwithin{equation}{section}

\theoremstyle{definition}

\theoremstyle{theorem}
\newtheorem{theoremalpha}{Theorem}
\theoremstyle{corollary}

\newtheorem{theorem}{Theorem}[section]
\newtheorem{main theorem}{Main Theorem}
\newtheorem{proposition}[theorem]{Proposition}

\newtheorem{corollary}[theorem]{Corollary}

\newtheorem{lemma}[theorem]{Lemma}

\newtheorem{theorem*}{Theorem}
\newtheorem{corollary*}[theorem*]{Corollary}
\newtheorem{conjecture*}[theorem*]{Conjecture}

\theoremstyle{definition}
\newtheorem{example}[theorem]{Example}
\newtheorem{definition}[theorem]{Definition}
\newtheorem{definition-lemma}[theorem]{Definition-Lemma}

\newtheorem{remark}[theorem]{Remark}

\newcommand\R{\mathbb{R}}
\newcommand\Q{\mathbb{Q}}
\newcommand\Z{\mathbb{Z}}

\renewcommand\P{\mathbb{P}}

\newcommand\eps{\varepsilon}
\renewcommand\epsilon{\varepsilon}

\newcommand{\mc}{\mathcal}

\newcommand{\ol}[1]{\overline{#1}}

\DeclareMathOperator{\ord}{ord}
\DeclareMathOperator{\mult}{mult}

\DeclareMathOperator{\Supp}{Supp}

\DeclareMathOperator{\Cent}{Cent}
\DeclareMathOperator{\vol}{vol}
\DeclareMathOperator{\val}{val}

\DeclareMathOperator{\Eff}{Eff}

\DeclareMathOperator{\Pic}{Pic}
\DeclareMathOperator{\SB}{SB}

\newcommand{\bm}{\mathbf B_-}  

\newcommand{\bp}{\mathbf B_+}  
\newcommand{\okbd}{\Delta}
\newcommand{\okval}{\Delta^{\val}}
\newcommand{\oklim}{\Delta^{\lim}}
\newcommand{\oknum}{\Delta^{\text{num}}}

\def\Cent{\operatorname{Cent}}

\def\BDPP{\operatorname{BDPP}}

\begin{document}

\title[Okounkov bodies associated to abundant divisors and Iitaka fibrations]{Okounkov bodies associated to abundant divisors \\and Iitaka fibrations}

\author{Sung Rak Choi}
\address{Department of Mathematics, Yonsei University, Seoul, Korea}
\email{sungrakc@yonsei.ac.kr}

\author{Jinhyung Park}
\address{Department of Mathematics, Sogang University, Seoul, Korea}
\email{parkjh13@sogang.ac.kr}

\author{Joonyeong Won}
\address{Center for Mathematical Challenges, Korea Institute for Advanced Study, Seoul, Korea}
\email{leonwon@kias.re.kr}

\date{\today}
\keywords{Okounkov body, abundant divisor, Fujita's approximation, Iitaka fibration}

\begin{abstract}
The aim of this paper is to study the Okounkov bodies associated to abundant divisors. As a main result, we prove that the valuative Okounkov bodies of an abundant divisor encode all the numerical properties. We apply this result to recover the asymptotic base loci of an abundant divisor from the valuative Okounkov bodies. We also give a criterion of when the valuative and limiting Okounkov bodies of an abundant divisor coincide by comparing their Euclidean volumes. To obtain these results, we prove some variants of Fujita's approximations for Okounkov bodies using Iitaka fibrations.
\end{abstract}

\maketitle


\section{Introduction}

Inspired by the work of Okounkov \cite{O1, O2}, Lazarsfeld--Musta\c{t}\u{a} \cite{lm-nobody} and Kaveh--Khovanskii \cite{KK} independently introduced and studied the convex sets called the Okounkov bodies associated to big divisors. Following their philosophy, there has been a number of attempts to understand the various asymptotic properties of divisors by analyzing the structure of the Okounkov bodies.
The details are as follows. We first let $X$ be a smooth projective variety of dimension $n$. For a divisor $D$ on $X$, the \emph{Okounkov body} $\okbd_{Y_\bullet}(D)$ is defined as a convex set in $\mathbb R^n$ which clearly depends on $D$ and also on the choice of the admissible flag $Y_\bullet$ (see Definition \ref{def-okbd}). It is expected that one can extract various positivity properties of the divisor $D$ from the structure of the Okounkov bodies. Based on the results on the Okounkov bodies of big divisors \cite{lm-nobody}, we extended in \cite{CHPW1, CPW1, CPW2} the study of Okounkov bodies to pseudoeffective divisors by introducing the \emph{valuative Okounkov body}  $\okval_{Y_\bullet}(D)$ and the \emph{limiting Okounkov body} $\oklim_{Y_\bullet}(D)$ of a pseudoeffective divisor $D$ (see Definition \ref{def-vallim}). By definition, $\okval_{Y_\bullet}(D)\subseteq\oklim_{Y_\bullet}(D)$ holds in general and $\okbd_{Y_\bullet}(D)=\okval_{Y_\bullet}(D)=\oklim_{Y_\bullet}(D)$ when $D$ is a big divisor. See Subsection \ref{okbdsubsec} for more details.

\medskip

By \cite[Proposition 4.1 (i)]{lm-nobody} and \cite[Theorem A]{Jow}, it is known that the Okounkov bodies are numerical in nature, i.e., two big divisors $D, D'$ on a smooth projective variety $X$ are numerically equivalent if and only if $\okbd_{Y_\bullet}(D)=\okbd_{Y_\bullet}(D')$ for every admissible flag $Y_\bullet$ on $X$. This statement was extended to pseudoeffective divisors using the limiting Okounkov bodies in \cite[Theorem C]{CHPW1}. Thus theoretically one could read off all the numerical information of a given pseudoeffective divisor from its limiting Okounkov bodies. In contrasts, the valuative Okounkov bodies do not reflect the numerical properties of divisors in full as we observed in \cite[Remark 3.13]{CHPW1}.

\medskip

The first aim of the paper is to show that as is often the case, imposing the ``abundance condition'' on divisors turns the valuative Okounkov bodies into numerical objects. In this paper, following \cite{BDPP}, \cite{CP}, we say that a divisor $D$ is \emph{abundant} if $\kappa(D)=\nu_{\BDPP}(D)$ holds. Since $\kappa(D)\leq\nu_{\BDPP}(D)\leq \kappa_{\sigma}(D) \leq \kappa_{\nu}(D)$ holds in general, our definition is weaker than the classical abundance which requires $\kappa(D)=\kappa_{\sigma}(D)$ or $\kappa(D)=\kappa_{\nu}(D)$. We refer to Subsection \ref{iitakadimsubsec} for the definitions of numerical Iitaka dimensions $\nu_{\BDPP}(D)$, $\kappa_{\sigma}(D)$, $\kappa_{\nu}(D)$ and to Subsection \ref{absubsec} for abundant divisors.

\medskip

The following theorem is an extension of \cite[Proposition 4.1 (i)]{lm-nobody} and \cite[Theorem A]{Jow} to valuative Okounkov bodies of abundant divisors.

\begin{theoremalpha}[=Corollary \ref{numeq=same}]\label{main1}
Let $D, D'$ be pseudoeffective abundant $\R$-divisors on a smooth projective variety $X$. Then we have:
$$
D \equiv D' ~~\text{ if and only if }~~ \okval_{Y_\bullet}(D)=\okval_{Y_\bullet}(D')~\text{ for every admissible flag $Y_\bullet$ on $X$}.
$$
\end{theoremalpha}

We remark that the `only if' direction of Theorem \ref{main1} does not hold when $D, D'$ are not abundant. It is because $\dim \okval_{Y_\bullet}(D'')=\kappa(D'')$ holds for any divisor $D''$ while we may possibly have $\kappa(D) \neq \kappa(D')$ even when $D \equiv D'$ (see \cite[Remark 3.13]{CHPW1}). However, the `if' direction of Theorem \ref{main1} holds without the abundance assumption on $D, D'$ (see Proposition \ref{same=>numeq}). As a consequence, we will also show in Corollary \ref{lineq=same} that if $\Pic(X)$ is finitely generated, then for any divisors $D, D'$ with $\kappa(D), \kappa(D') \geq 0$, we have:
\emph{
$$
D \sim_{\R} D' ~~\text{ if and only if }~~ \okval_{Y_\bullet}(D)=\okval_{Y_\bullet}(D')~\text{ for every admissible flag $Y_\bullet$ on $X$}.
$$}\\[-9pt]
\indent It is natural to ask how to extract the numerical properties of abundant divisors from the valuative Okounkov bodies. To give a partial answer to this question, we study the restricted base locus $\bm(D)$ (see Subsection \ref{asysubsec} for the definition) of an abundant divisor $D$ using the valuative Okounkov bodies. The analogous results of the following theorem for limiting Okounkov bodies was obtained in \cite[Theorem A]{CHPW2} (see also \cite{AV-loc pos2, AV-loc pos3, AV-loc pos}).

\begin{theoremalpha}[=Theorem \ref{b-thm}]\label{main2}
Let $D$ be a pseudoeffective abundant $\R$-divisor on a smooth projective variety $X$ of dimension $n$, and $x \in X$ be a point. Then the following are equivalent:
\begin{enumerate} [leftmargin=0cm,itemindent=.6cm]
 \item[$(1)$] $x \in \bm(D)$.
 \item[$(2)$] $\okval_{Y_\bullet}(D)$ does not contain the origin of $\R^n$ for every admissible flag $Y_\bullet$ on $X$ centered at $x$.
  \item[$(3)$] $\okval_{Y_\bullet}(D)$ does not contain the origin of $\R^n$ for some admissible flag $Y_\bullet$ on $X$ centered at $x$.
\end{enumerate}
\end{theoremalpha}

As we observed in \cite[Remark 4.10]{CPW1}, without the abundance condition, $\okval_{Y_\bullet}(D)$ may not contain the origin of $\R^n$ for some admissible flag $Y_\bullet$ even if $D$ is nef. Note that the analogous statements concerning $\bp(D)$ as in \cite[Theorem C]{CHPW2} for an abundant divisor $D$ easily follow from \cite[Theorem 6.5]{CHPW2} since big divisors are abundant and $\bp(D)=X$ holds if $D$ is not big.

\medskip

In \cite{CHPW1}, we have seen that the Okounkov bodies $\okval_{Y_\bullet}(D)$ and $\oklim_{Y_\bullet}(D)$ encode a good amount of asymptotic properties of the divisor $D$ if the given admissible flag $Y_\bullet$ contains a Nakayama subvariety or a positive volume subvariety of $D$, respectively (see Subsection \ref{okbdsubsec} for the definitions of these special subvarieties). For example, we have $\dim \okval_{Y_\bullet}(D) = \kappa(D)$ and $\dim \oklim_{Y_\bullet}(D)=\nu_{\BDPP}(D)$ for such admissible flags $Y_\bullet$.
Thus for the two Okounkov bodies $\oklim_{Y_\bullet}(D)$ and $\okval_{Y_\bullet}(D)$ to coincide with each other, it is necessary to assume that $\kappa(D)=\nu_{\BDPP}(D)$, i.e., $D$ is an abundant divisor. In this case, we show in Proposition \ref{nak=psv}  that a subvariety is a Nakayama subvariety of $D$ if and only if it is a positive volume subvariety of $D$.
However, even under the abundance condition, the inclusion $\okval_{Y_\bullet}(D)\subseteq\oklim_{Y_\bullet}(D)$ can be strict as was noticed in \cite[Example 4.2]{CHPW1}. By comparing the Euclidean volumes of the Okounkov bodies $\okval_{Y_\bullet}(D)$ and $\oklim_{Y_\bullet}(D)$, we obtain a criterion for the equality of these bodies.

\begin{theoremalpha}[=Theorem \ref{comvallim}]\label{main3}
Let $D$ be a pseudoeffective abundant $\R$-divisor on an $n$-dimensional smooth projective variety $X$ with $\kappa(D)>0$. Fix an admissible flag $Y_\bullet$ on $X$ such that $V=Y_{n-\kappa(D)}$ is a Nakayama subvariety of $D$ and $Y_n$ is a general point in $X$.
Consider the Iitaka fibration $\phi \colon X' \to Z$ of $D$ and the strict transform $V'$ of $V$ on $X'$. Then we have
$$
\vol_{\R^{\kappa(D)}}(\oklim_{Y_\bullet}(D))=\deg(\phi|_{V'} \colon V' \to Z) \cdot \vol_{\R^{\kappa(D)}}(\okval_{Y_\bullet}(D)).
$$
In particular, $\okval_{Y_\bullet}(D)=\oklim_{Y_\bullet}(D)$ if and only if the map $\phi|_{V'} \colon V' \to Z$ is generically injective.
\end{theoremalpha}

We remark that even if $D$ is an abundant $\R$-divisor with $\kappa(D)>0$, there may not exist Nakayama subvarieties $V$ giving rise to a generically injective map $\phi|_{V'} \colon V' \to Z$ (see Example \ref{nonexist}). See also Section 4 of \cite{CHPW1} for more related results.

\medskip

To prove all the above theorems, we use the results on Nakayama subvarieties and Iitaka fibrations (see Subsection \ref{absubsec}). Another key ingredients are some versions of Fujita's approximations for the valuative Okounkov bodies $\okval_{Y_\bullet}(D)$ of an effective divisor $D$ (Lemma \ref{altconval}) and for the limiting Okounkov bodies $\oklim_{Y_\bullet}(D)$ of an abundant divisor $D$ (Lemma  \ref{altconlim}). These results may be also regarded as alternative constructions of Okounkov bodies $\okval_{Y_\bullet}(D)$ and $\oklim_{Y_\bullet}(D)$.

\medskip
The organization of the paper is as follows. We begin by collecting relevant basic facts on various asymptotic invariants, Iitaka fibrations, Zariski decompositions, Okounkov bodies, numerical Iitaka dimensions, etc. in Section \ref{prelimsec}. In Section \ref{altsec}, we prepare the main ingredients required for the proofs of Theorems \ref{main1} and \ref{main3}.  Sections \ref{jowsec}, \ref{b-sec}, and \ref{compsec} are devoted to proving Theorems \ref{main1}, \ref{main2}, and \ref{main3}, respectively.

\section{Preliminaries}\label{prelimsec}

In this section, we collect relevant facts which will be used later.

\subsection{Conventions}
Throughout the paper, we work over the field $\mathbb C$ of complex numbers.
By a \emph{$($sub$)$variety}, we mean an irreducible (sub)variety, and $X$ denotes a smooth projective variety of dimension $n$. Unless otherwise stated, a \emph{divisor} means an $\R$-Cartier $\R$-divisor. A divisor $D$ on $X$ is \emph{pseudoeffective} if its numerical class $[D] \in N^1(X)_{\R}$ lies in the pseudoeffective cone $\ol\Eff(X)$, the closure of the cone spanned by effective divisor classes.
A divisor $D$ on $X$ is \emph{big} if $[D]$ lies in the interior $\text{Big}(X)$ of $\ol\Eff(X)$.

\subsection{Asymptotic invariants}\label{asysubsec}
Let $D$ be a divisor on $X$. The \emph{stable base locus} of $D$ is defined as $\SB(D):= \bigcap_{D \sim_{\R} D' \geq 0} \Supp(D')$.
The \emph{augmented base locus} of $D$ is defined as $\bp(D):=\bigcap_A\text{SB}(D-A)$
where the intersection is taken over all ample divisors $A$.
The \emph{restricted base locus} of $D$ is defined as $\bm(D):=\bigcup_{A}\SB(D+A)$
where the union is taken over all ample divisors $A$.
It is well known that  $\bp(D)$ and $\bm(D)$ depend only on the numerical class of $D$. Note that $\bm(D)=X$  (resp. $\bp(D)=X$) if and only if $D$ is not pseudoeffective (resp. not big), and $\bm(D)=\emptyset$ (resp. $\bp(D)=\emptyset$) if and only if $D$ is nef (resp. ample).
For more details, see \cite{elmnp-asymptotic inv of base}.

\medskip

Consider a subvariety $V \subseteq X$ of dimension $v$.
The \emph{restricted volume} of $D$ along $V$ is defined as
$ \displaystyle \vol_{X|V}(D):=\limsup_{m \to \infty} \frac{h^0(X|V,\lfloor mD \rfloor)}{m^v/v!}$
where $h^0(X|V, \lfloor mD \rfloor)$ is the dimension of the image of the natural restriction map $H^0(X, \mc O_X(\lfloor mD \rfloor ))\to H^0(V,\mc O_V(\lfloor mD \rfloor|_V))$.
If $V \not\subseteq \bp(D)$, then the restricted volume $\vol_{X|V}(D)$ depends only on the numerical class of $D$, and it uniquely extends to a continuous function $\vol_{X|V} \colon \text{Big}^V (X) \to \R$ where $\text{Big}^V(X)$ is the set of all $\R$-divisor classes $\xi$ such that $V$ is not properly contained in any irreducible component of $\bp(\xi)$. When $V=X$, we simply let $\vol_X(D):=\vol_{X|X}(D)$, and we call it the \emph{volume} of $D$.
For more details, we refer to \cite[Section 2.2.C]{pos}, \cite{elmnp-restricted vol and base loci}.

\medskip

Now, assume that $V\not\subseteq\bm(D)$.
The \emph{augmented restricted volume} of $D$ along $V$ is defined as
$\vol_{X|V}^+(D):=\lim_{\eps\to 0+} \vol_{X|V}(D+\eps A)$
where $A$ is an ample divisor on $X$. The definition is independent of the choice of $A$.
Note that $\vol_{X|V}^+(D)=\vol_{X|V}(D)$ for $D \in \text{Big}^V (X)$.
This also extends uniquely to a continuous function $\vol_{X|V}^+ \colon \overline{\text{Eff}}^V(X) \to \R$ where $\overline{\text{Eff}}^V(X) := \text{Big}^V(X) \cup \{ \xi \in \overline{\text{Eff}}(X) \setminus \text{Big}(X) \mid V \not\subseteq \bm(\xi)   \}$.
For $D\in \overline{\text{Eff}}^V(X)$, we have $\vol_{X|V}(D) \leq \vol_{X|V}^+(D) \leq \vol_{V}(D|_V)$, and both inequalities can be strict in general.
For more details, see \cite[Subsection 2.3]{CHPW1}.

\subsection{Iitaka fibration}\label{iitakasubsec}
Let $D$ be a divisor on $X$.
The \emph{Iitaka dimension} of $D$ is defined as
$$
\kappa(D):=\max \left\{ k \in \Z_{\geq 0} \left|\; \limsup_{m \to \infty} \frac{h^0(X, \mathcal{O}_X(\lfloor mD \rfloor))}{m^k}>0 \right.\right\}
$$
if  $h^0(X, \mathcal{O}_X(\lfloor mD \rfloor)) \neq 0$ for some $m > 0$, and $\kappa(D):=-\infty$ otherwise.
Note that $\kappa(D)$ is not an invariant of the $\R$-linear equivalence class of $D$.
Nonetheless, it satisfies the property that $\kappa(D)=\kappa(D')$ when $\kappa(D), \kappa(D') \geq 0$ and $D\sim_{\R}D'$ (see \cite[Remark 2.8]{CHPW1}).

\medskip

Now, assume that  $\kappa(D) > 0$. Then there exists a morphism $\phi \colon X' \to Z$ between smooth projective varieties $X', Z$ with connected fibers such that for all sufficiently large integers $m>0$, the rational maps $\phi_{mD} \colon X \dashrightarrow Z_m$ defined by $|\lfloor mD \rfloor|$ are birationally equivalent to $\phi$, i.e., there exists a commutative diagram
\[
\begin{split}
\xymatrix{
 X \ar@{-->}[d]_{\phi_{mD}} & X' \ar[l]_{f} \ar[d]^{\phi} \\
 Z_m & Z \ar[l]^{g_m}
}
\end{split}
\]
of a rational map $\phi_{mD}$ and morphisms $f, \phi, g_m$ with connected fibers, where the horizontal maps $f, g_m$ are birational, $\dim Z = \kappa(D)$, and $\kappa(f^*D|_F)=0$, where $F$ is a very general fiber of $\phi$.
(see e.g., \cite[Theorem 2.1.33]{pos}, \cite[Theorem--Definition II.3.14]{nakayama}).
Such a fibration is called an \emph{Iitaka fibration} of $D$. It is unique up to birational equivalence.

\subsection{Divisorial Zariski decompositions}\label{zdsubsec}
To define the divisorial Zariski decomposition, we first consider a divisorial valuation $\sigma$ on $X$ with the center $V:=\Cent_X \sigma$ on $X$.
If $D$ is a big divisor on $X$, we define \emph{the asymptotic valuation} of $\sigma$ at $D$ as
$\ord_V(||D||):=\inf\{\sigma(D')\mid D\equiv D'\geq 0\}$.
If $D$ is only a pseudoeffective divisor on $X$, we define
$\displaystyle \ord_V(||D||):=\lim_{\epsilon\to 0+}\ord_V(||D+\eps A||)$
for some ample divisor $A$ on $X$. This definition is independent of the choice of $A$. Note that $\ord_V(||D||)$ is a numerical invariant of $D$.
The \emph{divisorial Zariski decomposition} of a pseudoeffective divisor $D$ is the decomposition
$$
D=P_{\sigma}+N_{\sigma} = P_{\sigma}(D) + N_{\sigma}(D)
$$
into the \emph{negative part}  $N_{\sigma}=N_{\sigma}(D):=\sum_{E}  \ord_E(||D||)E$
where the summation is over all the finitely many prime divisors $E$ of $X$ such that $ \ord_E(||D||)>0$ and the \emph{positive part} $P_{\sigma}=P_{\sigma}(D):=D-N_{\sigma}$. The positive part $P_\sigma(D)$ is characterized as the maximal divisor such that $P_\sigma \leq D$ and $P_\sigma(D)$ is movable (see \cite[Proposition III.1.14]{nakayama}). Note that by construction $N_\sigma(D)$ is a numerical invariant of $D$.
For more details, see \cite{B1}, \cite{nakayama}, \cite{P}.

\medskip

Let $D$ be a divisor on $X$ with $\kappa(D) \geq 0$.
The \emph{$s$-decomposition} of $D$ is the decomposition
$$
D=P_s+N_s = P_s(D) + N_s(D)
$$
into the \emph{negative part} $N_s=N_s(D):=\inf\{L \mid L \sim_{\R} D, L \geq 0\}$ and the \emph{positive part} $P_s=P_s(D):=D-N_s$.
The positive part $P_s(D)$ is characterized as the smallest divisor such that $P_s \leq D$ and $R(X, P_s) \simeq R(X, D)$ (see \cite[Proposition 4.8]{P}). Note that $N_s(D)$ is an $\R$-linear equivalence invariant of $D$.
Note that $P_s(D) \leq P_\sigma(D)$ and $P_s(D), P_\sigma(D)$ do not coincide in general. If $D$ is an abundant divisor (see Definition \ref{def-abundant}), then $P_s(D)=P_{\sigma}(D)$ so that $P_s(D), N_s(D)$ become numerical invariants of $D$ (see Theorem \ref{abprop} (2)).
For more details, see \cite{P}.

\subsection{Numerical Iitaka dimensions}\label{iitakadimsubsec}
Let $D$ be a pseudoeffective divisor on $X$. There are several notions of the numerical Iitaka dimensions in the literature defined from different perspectives (see e.g. \cite{BDPP}, \cite{CP}, \cite{E}, \cite{lehmann-nu, lehmann-red}, \cite{lesieutre}, \cite{nakayama}). Among them, the following dimension first introduced by Boucksom--Demailly--P\u{a}un--Peternell \cite{BDPP} is of our most interest:
$$
\nu_{\BDPP}(D):=\max\left\{k\in\mathbb Z_{\geq 0}\left|\langle D^k\rangle\neq 0\right.\right\}.
$$
Here $\langle D^k \rangle$ is the positive intersection product (see \cite[Section 4]{lehmann-nu} for the definition and basic properties). By \cite[Theorem 6.2]{lehmann-nu} (see also \cite[Theorem 1.1]{CP}), we have
$$
\begin{array}{rcl}
\nu_{\BDPP}(D) & = & \max\left\{ \dim W \left| \vol^+_{X|W}(L) > 0\right. \right\}\\
&=&\max \left\{ \dim W \left| \inf\limits_{\phi} \vol_{\widetilde{W}}(P_{\sigma}(\phi^* D)|_{\widetilde{W}}) > 0 \right.\right\}
\end{array}
$$
where $W$ range over all the irreducible subvarieties of $X$ not contained in $\bm(D)$, and $\phi \colon (\widetilde{X}, \widetilde{W}) \to (X, W)$ range over all $W$-birational models, which by definition mean that $W$ is not contained in any $\phi$-exceptional center and $\widetilde{W}$ is the strict transform of $W$ (see \cite[Definition 2.10]{lehmann-nu}). We have $\nu_{\BDPP}(D) \geq 0$ whenever $D$ is pseudoeffective. We put $\nu_{\BDPP}(D) := -\infty$ when $D$ is not pseudoeffective. We will use the following basic properties of $\nu_{\BDPP}(D)$:
\begin{enumerate}
 \item $\kappa(D) \leq \nu_{\BDPP}(D)$.
 \item $\nu_{\BDPP}(D) \leq n$, and $\nu_{\BDPP}(D)=n$ if and only if $D$ is big.
 \item $\nu_{\BDPP}(D)$ is a numerical invariant of $D$, i.e., $\nu_{\BDPP}(D)=\nu_{\BDPP}(D')$ whenever $D \equiv D'$.
 \item $\nu_{\BDPP}(D)=\nu_{\BDPP}(P_{\sigma}(D))$.
\end{enumerate}
We refer to \cite{CP}, \cite{lehmann-nu}, \cite{nakayama} for more detailed properties.

\medskip

We recall some other numerical Iitaka dimensions defined for a pseudoeffective $D$
$$
\arraycolsep=1.4pt\def\arraystretch{1.9}
\begin{array}{rl}
\kappa_{\sigma}(D)&:=\displaystyle  \max \left\{ k \in \Z_{\geq 0} \left| \limsup\limits_{m \to \infty} \frac{h^0(X, \lfloor mD \rfloor + A)}{m^k} > 0 \right.\right\} \\
\kappa_{\nu}(D)&:=\min \{ \dim W \mid \text{$D$ does not numerically dominate $W$}\}\\
\kappa_{\vol}(D)&:=\displaystyle \max \left\{k\in \Z_{\geq 0}\left| \liminf\limits_{\varepsilon \to 0} \frac{\vol_X(D+ \varepsilon A)}{\varepsilon^{n-k}} > 0\right. \right\}
\end{array}
$$
where $A$ is a sufficiently positive ample $\Z$-divisor on $X$. The first two dimensions $\kappa_{\sigma}$ and $\kappa_{\nu}$ were defined by Nakayama \cite[Chapter V]{nakayama} (see \cite[Definition V.2.22]{nakayama} for the definition of numerical dominance), and the third dimension $\kappa_{\vol}$ is defined by Lehmann \cite{lehmann-nu}. They are also numerical invariants of $D$. Note that $\kappa_{\sigma}(D), \kappa_{\vol}(D), \kappa_{\nu}(D)$ are nonnegative integers at most $n=\dim X$ when $D$ is pseudoeffective and $\kappa_{\sigma}(D), \kappa_{\vol}(D), \kappa_{\nu}(D) =n$ if and only if $D$ is big. By \cite[Proposition 3.1]{CP}, we have
$$
\nu_{\BDPP}(D) \leq \kappa_{\sigma}(D) \leq \kappa_{\nu}(D) \;\text{ and }\; \nu_{\BDPP}(D) \leq \kappa_{\vol}(D) \leq \kappa_{\nu}(D).
$$
It is worth noting that these numerical Iitaka dimensions $\kappa_{\sigma}(D), \kappa_{\vol}(D), \kappa_{\nu}(D)$ can be strictly larger than $\nu_{\BDPP}(D)$ (see \cite[Theorem 3]{lesieutre}, \cite[Theorem 1.2]{CP}).

\subsection{Okounkov bodies}\label{okbdsubsec}
Here we recall the construction and basic properties of Okounkov bodies associated to pseudoeffective divisors in \cite{lm-nobody}, \cite{KK}, \cite{CHPW1}. Throughout this subsection, we fix an \emph{admissible flag} $Y_\bullet$ on $X$, which by definition is a sequence of subvarieties
$$
Y_\bullet: X=Y_0\supseteq Y_1\supseteq\cdots \supseteq Y_{n-1}\supseteq Y_n=\{x\}
$$
where each $Y_i$ is an irreducible subvariety of codimension $i$ in $X$ and is nonsingular at $x$.
Let $D$ be a divisor on $X$ with $|D|_{\R}:=\{ D' \mid D \sim_{\R} D' \geq 0 \}\neq\emptyset$.
We define a valuation-like function
$$
\nu_{Y_\bullet} \colon |D|_{\R}\to \R_{\geq0}^n
$$
as follows:
for $D'\in |D|_\R$, let $\nu_1=\nu_1(D'):=\ord_{Y_1}(D')$.
Since $D'-\nu_1(D')Y_1$ is effective and does not contain $Y_2$ in the support, we define $\nu_2=\nu_2(D'):=\ord_{Y_2}((D'-\nu_1Y_1)|_{Y_1})$.
We then inductively define $\nu_{i+1}=\nu_{i+1}(D'):=\ord_{Y_{i+1}}((\cdots((D'-\nu_1Y_1)|_{Y_1}-\nu_2Y_2)|_{Y_2}-\cdots-\nu_iY_i)|_{Y_{i}})$.
Thus we finally obtain
$$
\nu_{Y_\bullet}(D'):=(\nu_1(D'),\nu_2(D'),\ldots,\nu_n(D')) \in \R_{\geq 0}^n
$$

\begin{definition}\label{def-okbd}
When $|D|_{\R} \neq \emptyset$, the \emph{Okounkov body} $\okbd_{Y_\bullet}(D)$ of a divisor $D$ on $X$ with respect to an admdissible flag $Y_\bullet$ on $X$ is defined as the closure of the convex hull of $\nu_{Y_\bullet}(|D|_{\R})$ in $\R^n_{\geq 0}$. When $|D|_{\R} = \emptyset$, we set $\okbd_{Y_\bullet}(D) := \emptyset$.
\end{definition}

More generally, a similar construction can be applied to a graded linear series $W_\bullet$ associated to a $\Z$-divisor on $X$ to construct the Okounkov body $\okbd_{Y_\bullet}(W_\bullet)$ of $W_\bullet$ with respect to $Y_\bullet$. For more details, we refer to \cite{lm-nobody}.

\medskip

In \cite{lm-nobody}, \cite{KK}, the Okounkov bodies $\okbd_{Y_\bullet}(D)$ were mainly studied for big divisors. When $D$ is not big, the following extension was introduced in \cite{CHPW1}.

\begin{definition}[{\cite[Definition 1.1]{CHPW1}}]\label{def-vallim}
\begin{enumerate}[leftmargin=0cm,itemindent=.6cm]
\item[(1)] For a divisor $D$ which is effective up to $\sim_\R$, i.e., $|D|_{\R}\neq \emptyset$, the \emph{valuative Okounkov body} $\okval_{Y_\bullet}(D)$ of $D$ with respect to an admissible flag $Y_\bullet$ is defined as the closure of the convex hull of $\nu_{Y_\bullet}(|D|_{\R})$ in $\R^n_{\geq 0}$.
If $|D|_\R=\emptyset$, then we set $\okval_{Y_\bullet}(D):=\emptyset$.
\item[(2)] For a pseudoeffective divisor $D$, the \emph{limiting Okounkov body} $\oklim_{Y_\bullet}(D)$ of $D$ with respect to an admissible flag $Y_\bullet$ is defined as
$$\oklim_{Y_\bullet}(D):=\lim_{\epsilon \to 0+}\okbd_{Y_\bullet}(D+\epsilon A) = \bigcap_{\epsilon >0} \okbd_{Y_\bullet}(D+\epsilon A) \subseteq \R^n_{\geq 0}$$
where $A$ is an ample divisor on $X$.
(Note that $\oklim_{Y_\bullet}(D)$ is independent of the choice of $A$.)
If $D$ is not pseudoeffective, then we set $\oklim_{Y_\bullet}(D)
:=\emptyset$.
\end{enumerate}\end{definition}

Note that we actually have $\okbd_{Y_\bullet}(D)=\okval_{Y_\bullet}(D)$ for any divisor $D$ and  any admissible flag $Y_\bullet$. However, we will only use the notation $\okval_{Y_\bullet}(D)$ when $D$ is known to be non-big or at least when the bigness of $D$ is not clear in order to distinguish our results from the well known cases for big divisors. We also remark that Boucksom's \emph{numerical Okounkov body} $\oknum_{Y_\bullet}(D)$ in \cite{B2} coincides with our limiting Okounkov body $\oklim_{Y_\bullet}(D)$.

\medskip

By construction, the valuative Okounkov body $\okval_{Y_\bullet}(D)$ is only an $\R$-linear invariant of $D$, not a numerical invariant of $D$ (see \cite[Remark 3.13 and Proposition 3.15]{CHPW1}). The limiting Okounkov body $\oklim_{Y_\bullet}(D)$ is a numerical invariant of $D$. More precisely, for pseudoeffective divisors $D,D'$, it is known that $D \equiv D'$ if and only if $\oklim_{Y_\bullet}(D)=\oklim_{Y_\bullet}(D')$ for every admissible flag $Y_\bullet$ on $X$ (see \cite[Theorem C]{CHPW1}).

\begin{lemma}[cf. {\cite[Lemma 3.4]{CHPW2}, \cite[Lemma 3.4]{CPW2}}]\label{okbdbir}
Let $D$ be a divisor on $X$. Consider a birational morphism $f : \widetilde{X} \to X$ with $\widetilde{X}$ smooth and an admissible flag
$$
\widetilde{Y}_\bullet : \widetilde{X}=\widetilde{Y}_0 \supseteq \widetilde{Y}_1 \supseteq \cdots \supseteq \widetilde{Y}_{n-1} \supseteq \widetilde{Y}_n=\{ x' \}
$$
on $\widetilde{X}$. Suppose that $f$ is isomorphic over a neighborhood of $f(x')$ and
$$
Y_\bullet:=f(\widetilde{Y}_\bullet) : X=f(\widetilde{Y}_0) \supseteq f(\widetilde{Y}_1) \supseteq \cdots \supseteq f(\widetilde{Y}_{n-1}) \supseteq f(\widetilde{Y}_n)=\{ f(x') \}
$$ is an admissible flag on $X$. Then
$\okval_{\widetilde{Y}_\bullet}(f^*D)=\okval_{Y_\bullet}(D)$ and
$\oklim_{\widetilde{Y}_\bullet}(f^*D) = \oklim_{Y_\bullet}(D)$.
\end{lemma}

\begin{proof}
The case of the limiting Okounkov body is shown in \cite[Lemma 3.4]{CHPW2}. The proof for the case of the valuative Okounkov body is almost identical, and we leave the details to the readers.
\end{proof}

\begin{remark}\label{smflag}
By Lemma \ref{okbdbir} and \cite[Lemma 3.5]{CHPW2}, we can assume that each $Y_i$ in the admissible flag $Y_\bullet$ on $X$ is smooth (see also \cite[Remark 3.6]{CHPW2}).
\end{remark}

\begin{lemma}[cf. {\cite[Lemma 3.9]{CHPW2}, \cite[Lemma 3.5]{CPW2}}]\label{okbdzd}
Let $D$ be a divisor on $X$ with the $s$-decomposition $D=P_s+N_s$ and the divisorial Zariski decomposition $D=P_\sigma+N_\sigma$. Fix an admissible flag $Y_\bullet$ on $X$. Then we have $\okval_{Y_\bullet}(D)=\okval_{Y_\bullet}(P_s) + \okval_{Y_\bullet}(N_s)$ and $\oklim_{Y_\bullet}(D)=\oklim_{Y_\bullet}(P_\sigma) + \oklim_{Y_\bullet}(N_\sigma)$. If $Y_n$ is a general point (i.e., $Y_n \not\subseteq  \Supp(N)$), then $\okval_{Y_\bullet}(D)=\okval_{Y_\bullet}(P_s)$ and $\oklim_{Y_\bullet}(D)=\oklim_{Y_\bullet}(P_\sigma)$.
\end{lemma}

\begin{proof}
The assertion for $\okval_{Y_\bullet}(D)$ follows from the fact that $R(X, D) \simeq R(X, P_s)$ and the construction of the valuative Okounkov body. The assertion for $\oklim_{Y_\bullet}(D)$ is nothing but \cite[Lemma 3.9]{CHPW2}.
\end{proof}

By definition, $\okval_{Y_\bullet}(D) \subseteq \oklim_{Y_\bullet}(D)$, and the inclusion can be strict in general (see \cite[Examples 4.2 and 4.3]{CHPW1}). If $D$ is big, then $\okbd_{Y_\bullet}(D)=\okval_{Y_\bullet}(D)=\oklim_{Y_\bullet}(D)$. For a divisor $D$ with $\kappa(D) \geq 0$, by \cite[Proposition 3.3]{B2} and \cite[Theorem 1.1]{CP}, we have
$$
\dim \okval_{Y_\bullet}(D) = \kappa (D) \leq \dim \oklim_{Y_\bullet}(D) \leq \nu_{\BDPP}(D)
$$
for any admissible flag $Y_\bullet$.

\begin{remark}\label{rem:correctionnumdim}
It was shown in \cite[Lemma 4.8]{B2} and \cite[Proof of Proposition 3.21]{CHPW1} that
$$
\dim \oklim_{Y_\bullet}(D) \leq \kappa_{\vol}(D)= \max \left\{k\in \Z_{\geq 0}\left| \liminf\limits_{\varepsilon \to 0} \frac{\vol_X(D+ \varepsilon A)}{\varepsilon^{n-k}} > 0\right. \right\}
$$
for every admissible flag $Y_\bullet$ on $X$. In \cite{CHPW1, CPW1, CPW2}, we use the coincidence of the numerical Iitaka dimensions $\kappa_{\vol}(D) = \nu_{\BDPP}(D)$, which was claimed in \cite{lehmann-nu}. However, based on Lesieutre's example in \cite{lesieutre}, Choi--Park proved that there exist smooth projective variety $Y$ and pseudoeffective divisor $E$ such that $\kappa_{\vol}(E) > \nu_{\BDPP}(E)$ (see \cite[Theorem 1.2]{CP}). Thus some results of \cite{CHPW1, CPW1, CPW2} are affected by these examples (in those papers, $\kappa_{\nu}$ is used to mean $\kappa_{\sigma}$, and is supposed to be equal to $\nu_{\BDPP}$ and $\kappa_{\vol}$). Fortunately, we have
$$
\nu_{\BDPP}(D)=\max\{ \dim \oklim_{Y_\bullet}(D) \mid Y_\bullet~\text{is an admissible flag on $X$} \},
$$
by \cite[Theorem 1.1]{CP}. If we use $\nu_{\BDPP}$ for the numerical Iitaka dimension, then all the results in \cite{CHPW1, CPW1, CPW2} are valid.
\end{remark}

In \cite{CHPW1}, we introduced a Nakayama subvariety and positive volume subvariety of a divisor $D$ to extract asymptotic invariants of $D$ from the Okounkov bodies.

\begin{definition}[{\cite[Definitions 2.12 and 2.19]{CHPW1}, \cite[Definition  4.1]{CPW2}}]
\begin{enumerate}[leftmargin=0cm,itemindent=.6cm]
\item[(1)] For a divisor $D$ such that $\kappa(D)\geq 0$, a \emph{Nakayama subvariety of $D$} is defined as an irreducible subvariety $U \subseteq X$ such that $\dim U=\kappa(D)$ and for every integer $m \geq 0$ the natural map
$$
H^0(X, \mc O_X(\lfloor mD \rfloor)) \to H^0(U, \mc O_U(\lfloor mD \rfloor|_U))
$$
is injective (or equivalently, $H^0(X, \mc I_U \otimes \mc O_X(\lfloor mD \rfloor))=0$ where $\mc I_U$ is an ideal sheaf of $U$ in $X$).
\item[(2)] For a divisor $D$ with $\nu_{\BDPP}(D)\geq 0$, a \emph{positive volume subvariety of $D$} is defined as an irreducible subvariety $V \subseteq X$ such that $\dim V = \nu_{\BDPP}(D)$ and $\vol_{X|V}^+(D)>0$.
\end{enumerate}
\end{definition}

We have the following characterization of a Nakayama subvariety and a positive volume subvariety in terms of Okounkov bodies.

\begin{theorem}[{\cite[Theorem 1.2]{CPW2}}]\label{nakpvscrit}
Let $D$ be a divisor on $X$. Fix an admissible flag $Y_\bullet$ such that $Y_n$ is a general point in $X$. Then we have the following:
\begin{enumerate}[leftmargin=0cm,itemindent=.6cm]
\item[$(1)$] If $D$ is effective, then $Y_\bullet$ contains a Nakayama subvariety of $D$ if and only if $\okval_{Y_\bullet}(D) \subseteq \{0 \}^{n-\kappa(D)} \times \R^{\kappa(D)}$.
\item[$(2)$] If $D$ is pseudoeffective, then $Y_\bullet$ contains a positive volume subvariety of $D$ if and only if $\oklim_{Y_\bullet}(D) \subseteq \{0 \}^{n-\nu_{\BDPP}(D)} \times \R^{\nu_{\BDPP}(D)}$ and $\dim \oklim_{Y_\bullet}(D)=\nu_{\BDPP}(D)$.
\end{enumerate}
\end{theorem}

The following is the main result of \cite{CHPW1}.

\begin{theorem}[{\cite[Theorems A and B]{CHPW1}}]\label{chpwmain}
\begin{enumerate}[leftmargin=0cm,itemindent=.6cm]
\item[$(1)$] Let $D$ be a divisor on $X$ with $\kappa(D) \geq 0$. Fix an admissible flag $Y_\bullet$ containing a Nakayama subvariety $U$ of $D$ such that $Y_n$ is a general point in $X$. Then $\okval_{Y_\bullet}(D) \subseteq \{0 \}^{n-\kappa(D)} \times \R^{\kappa(D)}$ so that one can regard $\okval_{Y_\bullet}(D) \subseteq \R^{\kappa(D)}$. Furthermore, we have
$$\dim \okval_{Y_\bullet}(D)=\kappa(D) \text{ and } \vol_{\R^{\kappa(D)}}(\okval_{Y_\bullet}(D))=\frac{1}{\kappa(D)!} \vol_{X|U}(D).$$
\item[$(2)$] Let $D$ be a pseudoeffective divisor on $X$, and fix an admissible flag $Y_\bullet$ containing a positive volume subvariety $V$ of $D$. Then $\oklim_{Y_\bullet}(D) \subseteq \{0 \}^{n-\nu_{\BDPP}(D)} \times \R^{\nu_{\BDPP}(D)}$ so that one can regard $\oklim_{Y_\bullet}(D) \subseteq \R^{\nu_{\BDPP}(D)}$. Furthermore, we have
$$\dim \oklim_{Y_\bullet}(D)=\nu_{\BDPP}(D) \text{ and } \vol_{\R^{\nu_{\BDPP}(D)}}(\oklim_{Y_\bullet}(D))=\frac{1}{\nu_{\BDPP}(D)!} \vol_{X|V}^+(D).$$
\end{enumerate}
\end{theorem}

\begin{remark}
As in \cite{CHPW1}, \cite{CPW2}, when considering $\okval_{Y_\bullet}(D)$ (resp. $\oklim_{Y_\bullet}(D)$), we say that $Y_n$ is \emph{general} if it is not contained in $\text{SB}(D)$ (resp. $\bm(D)$) (see \cite[Remark 4.7]{CPW2}).
\end{remark}

The relation between the valuative Okounkov bodies and restricted volumes are also studied in \cite{DP}.

\subsection{Abundant divisor}\label{absubsec}
In this paper, we adopt the following notion of abundance.

\begin{definition}\label{def-abundant}
A pseudoeffective divisor $D$ on $X$ is said to be \emph{abundant} if $\kappa(D)=\nu_{\BDPP}(D)$ holds.
\end{definition}

We will need the following generalization of the well-known result of Kawamata for nef and abundant divisors \cite[Proposition 2.1]{kawa85} (see also the Errata of \cite{lehmann-red}).

\begin{theorem}[{\cite[Theorem 1.4]{CP}}]\label{thrm-abundant map}
Let $D$ be an effective $\R$-divisor on $X$ with $\kappa(D) > 0$. Then $D$ is abundant in the sense that $\kappa(D) = \nu_{\BDPP}(D)$ holds if and only if there are a birational morphism $\mu \colon W \to X$ from a smooth projective variety $W$ and a surjective morphism $g \colon W \to T$ to a smooth projective variety $Z$ with connected fibers such that
$$
P_{\sigma}(\mu^*D) \sim_{\Q} P_{\sigma}(g^*B)
$$
for some big divisor $B$ on $T$, where $g \colon W \to T$ is a birational model of the Iitaka fibration of $D$.
\end{theorem}

The following theorem essentially due to Lehmann will play crucial roles in proving our main results, Theorems \ref{main1}, \ref{main2}, and \ref{main3}.

\begin{theorem}\label{abprop}
Let $D$ be a pseudoeffective abundant divisor on $X$.
Then the following numerical properties hold:
\begin{enumerate} [leftmargin=0cm,itemindent=.6cm]
 \item[$(1)$] If $D'$ is a divisor on $X$ such that $\kappa(D') \geq 0$ and $D \equiv D'$, then $D'$ is also an abundant divisor.
 \item[$(2)$] For any divisorial valuation $\sigma$ on $X$ with the center $V=\Cent_X \sigma$ on $X$, we have
 $$
 \ord_V(||D||)= \inf\{ \sigma(D') \mid D \sim_{\R} D' \geq 0 \}.
 $$
 In particular, $P_{\sigma}(D)=P_s(D)$.
 \end{enumerate}
\end{theorem}
\begin{proof}
For $(1)$, we note that if $\kappa(D') \geq 0$, then $D'$ is $\Q$-linearly equivalent to an effective divisor. Thus $(1)$ follows from \cite[Corollary 6.3]{lehmann-red}.
For $(2)$, we apply \cite[Proposition 6.4]{lehmann-red} and \cite[Lemma 2.3]{CPW2}.
Note that the condition (5) of \cite[Theorem 6.1]{lehmann-red}, which is asserted in Theorem \ref{thrm-abundant map}, is used in the proofs of \cite[Corollary 6.3]{lehmann-red} and \cite[Proposition 6.4]{lehmann-red}.
\end{proof}

\begin{lemma}\label{notinsb}
Let $D$ be a pseudoeffective abundant divisor on $X$. If $V$ is a Nakayama subvariety of $D$ or a positive volume subvariety of $D$, then $V \not\subseteq \SB(D)$.
\end{lemma}

\begin{proof}
If $V$ is a Nakayama subvariety of $D$, then the assertion follows from definition. Assume that $V$ is a positive volume subvariety of $D$. We can take an admissible flag $Y_\bullet$ containing $V$. By Theorem \ref{chpwmain} (2), $\oklim_{Y_\bullet}(D) \subseteq \{ 0 \}^{n-\nu_{\BDPP}(D)} \times \R^{\nu_{\BDPP}(D)}$. Since $\okval_{Y_\bullet}(D) \subseteq \oklim_{Y_\bullet}(D)$, it follows that $\ord_{V}(D')=0$ for every effective divisor $D' \sim_{\R} D$. Thus $V \not\subseteq \Supp(D')$. Since $\text{SB}(D)\subseteq\Supp (D')$, we are done.
\end{proof}

\begin{proposition}\label{nak=psv}
Let $D$ be a pseudoeffective abundant divisor on $X$. A subvariety $V$ of $X$ is a Nakayama subvariety of $D$ if and only if it is a positive volume subvariety of $D$.
\end{proposition}

\begin{proof}
We can always construct an admissible flag $Y_\bullet$ on $X$ containing a given Nakayama subvariety of $D$ or a given positive volume subvariety of $D$. By Lemma \ref{notinsb},  we may assume that $Y_n$ is a general point in $X$.
Let $V\subseteq X$ be a Nakayama subvariety of $D$. By Theorem \ref{nakpvscrit} (1),  $\okval_{Y_\bullet}(D)\subseteq\{0\}^{n-\kappa(D)}\times\R^{\kappa(D)}$ for any admissible flag $Y_\bullet$ containing $V$ such that $Y_n$ is a general point in $X$.
Recall now that  $\okval_{Y_\bullet}(D)\subseteq\oklim_{Y_\bullet}(D)$ and $\dim\okval_{Y_\bullet}(D)=\kappa(D)=\nu_{\BDPP}(D)=\dim\oklim_{Y_\bullet}(D)$. Thus $\oklim_{Y_\bullet}(D)\subseteq\{0\}^{n-\kappa(D)}\times\R^{\kappa(D)}$.
Theorem \ref{nakpvscrit} (2) implies that $V$ is a positive volume subvariety of $D$.

Now, let $V\subseteq X$ be a positive volume subvariety of $D$, and $Y_\bullet$ be an admissible flag containing $V$ such that $Y_n$ is a general point in $X$. By Theorem \ref{nakpvscrit} (2), we have $\okval_{Y_\bullet}(D)\subseteq\oklim_{Y_\bullet}(D)\subseteq\{0\}^{n-\kappa(D)}\times\R^{\kappa(D)}$.
Theorem \ref{nakpvscrit} (1) implies that $V$ is a Nakayama subvariety of $D$.
\end{proof}

\section{Fujita's approximations for Okounkov bodies}\label{altsec}

The aim of this section is to prove some versions of Fujita's approximations for Okounkov bodies, which may be regarded as alternative constructions of valuative and limiting Okounkov bodies (see Lemmas \ref{altconval} and \ref{altconlim}). This will be used in the course of the proofs of Theorems \ref{main1} and \ref{main3}. Throughout the section, $X$ is a smooth projective variety of dimension $n$.

\subsection{Valuative Okounkov body case}\label{altsubsec1}
We fix notations used throughout this subsection.
Let $D$ be a divisor on $X$ with $\kappa(D) > 0$. We do not impose the abundant condition on $D$ in this subsection.
Fix an admissible flag $Y_\bullet$ on $X$ containing a Nakayama subvariety $U$ of $D$ such that $Y_n=\{x\}$ is general in $X$ so that $x \not\in \SB(D)$.
We can regard the valuative Okounkov body $\okval_{Y_\bullet}(D) \subseteq \{0\}^{n-\kappa(D)}\times\R^{\kappa(D)}$ as a subset of $\R^{\kappa(D)}$ (see Theorem \ref{chpwmain}).

\medskip

Now, for a sufficiently large integer $m>0$, we take a log resolution $f_m \colon X_m \to X$ of the base ideal $\frak{b}(\lfloor mD \rfloor)$ so that we obtain a decomposition $f_m^*(\lfloor mD \rfloor)=M_m' + F_m'$ into a base point free divisor $M_m'$ and the fixed part $F_m'$ of $|f_m^*(\lfloor mD \rfloor)|$. Let $M_m:=\frac{1}{m}M_m'$ and $F_m:=\frac{1}{m}F_m'$.
We may assume that $f_m \colon X_m \to X$ is isomorphic over a neighborhood of $x$.
Let $f_m^*D=P_m+N_m$ be the $s$-decomposition.

\medskip

Since $Y_n$ is general, by taking the strict transforms $Y_i^m$  of $Y_i$ on $X_m$, we obtain an admissible flag $Y_{\bullet}^m : Y_0^m \supseteq \cdots \supseteq Y_n^m $ on $X_m$.
We note that  $U_m:=Y_{n-\kappa(D)}^m$ is also a Nakayama subvariety of $f^*_mD$ since $f_m$ is $U$-birational (see \cite[Proposition 2.15]{CHPW1}). By definition, we see that $U_m$ is also a Nakayama subvariety of $M_m$.

\medskip

Let $W_\bullet$ be a graded linear series on $U$ associated to $D|_U$ where $W_k$ is the image of the natural injective map $H^0(X, \lfloor kD \rfloor) \to H^0(U, \lfloor kD \rfloor |_U)$.
We also consider a graded linear series $W_\bullet^m$ on $U_m$ associated to $M_m|_{U_m}$ where $W_k^m$ is the image of the natural injective map $H^0(X_m, \lfloor kM_m \rfloor) \to H^0(U_m, \lfloor kM_m \rfloor|_{U_m})$.
Note that $\dim W_m = \dim W_m^m$.
Let $\phi_m \colon X_m \to Z_m$ be the morphism defined by $|M_m'|$.
Then there is an ample divisor $H_m$ on $Z_m$ such that $\phi_m^*H_m = M_m$.
Note that $\phi_m|_{U_m} \colon U_m \to Z_m$ is a surjective morphism of projective varieties of the same dimension $\kappa(D)$.
Since $Y_n$ is general, we can assume that $\overline{Y}_\bullet^m : Z_m= \phi_m(Y_{n-\kappa(D)}^m)  \supseteq \cdots \supseteq \phi_m(Y_n^m)$ is an admissible flag on $Z_m$.

\medskip

The following lemma is the main result of this subsection.

\begin{lemma}\label{altconval}
Under the same notations as above, we have
$$\okval_{Y_\bullet}(D) = \lim_{m \to \infty} \okval_{Y_{\bullet}^m}(M_m)= \lim_{m \to \infty} \okbd_{\overline{Y}_\bullet^m}(H_m).$$
\end{lemma}

\begin{proof}
As we noted above, we treat $\okval_{Y_\bullet}(D),  \okval_{Y_{\bullet}^m}(M_m),$ and $\okbd_{\overline{Y}_\bullet^m}(H_m)$ as the subsets of the same fixed space $\R^{\kappa(D)}$.
By Lemmas \ref{okbdbir} and \ref{okbdzd}, we have
$$
\okval_{Y_\bullet}(D)=\okval_{Y_\bullet^m}(f_m^*D)=\okval_{Y_\bullet^m}(P_m),
$$
and by \cite[Remark 3.11]{CHPW1} and \cite[Lemma 5.1]{CPW2}, we have
$$
\okval_{Y_\bullet^m}(M_m)=\okbd_{Y^m_{n-\kappa(D)\bullet}}(W_\bullet^m)=\okbd_{\overline{Y}_\bullet^m}(H_m).
$$
Note that $\okval_{Y_\bullet^m}(M_m) \subseteq \okval_{Y_\bullet^m}(P_m)$.
By \cite[Remark 3.11]{CHPW1}, we also have
$$
\okval_{Y_\bullet}(D) = \okbd_{Y_{n-\kappa(D)\bullet}}(W_\bullet).
$$
By applying \cite[Remark 2.8, Theorems 2.13 and 3.3]{lm-nobody}, we see that
$$
\vol_{\R^{\kappa(D)}}(\okbd_{Y_{n-\kappa(D)\bullet}}(W_\bullet)) =
\lim_{m \to \infty} \vol_{\R^{\kappa(D)}}(\okbd_{Y^m_{n-\kappa(D)\bullet}}(W_\bullet^m)).
$$
As $\okbd_{Y^m_{n-\kappa(D)\bullet}}(W_\bullet^m) \subseteq \okbd_{Y_{n-\kappa(D)\bullet}}(W_\bullet)$, we obtain
$$
\okbd_{Y_{n-\kappa(D)\bullet}}(W_\bullet) = \lim_{m \to \infty} \okbd_{Y^m_{n-\kappa(D)\bullet}}(W_\bullet^m).
$$
Thus the assertion now follows.
\end{proof}

\begin{remark}
When $D$ is a big divisor, Lemma \ref{altconval} is the same as \cite[Theorem D]{lm-nobody}. See \cite[Remark 3.4]{lm-nobody} for the explanation on how this statement implies the classical statement of Fujita's approximation (see also \cite[Theorem 11.4.4]{pos}). Another version of Fujita's approximation for  effective divisors is stated in \cite[Theorem 1.2]{DP}.
\end{remark}

\subsection{Limiting Okounkov body case}\label{altsubsec2}
We fix notations used throughout this subsection.
Let $D$ be a pseudoeffective abundant divisor on $X$ with $\kappa(D)=\nu_{\BDPP}(D) > 0$.
Fix an admissible flag $Y_\bullet$ on $X$ containing a positive volume subvariety $V$ of $D$ such that $Y_n=\{x \}$ is general in $X$ so that $x \not\in \SB(D)$. By Proposition \ref{nak=psv}, $V$ is also a Nakayama subvariety of $D$.
We can regard the limiting Okounkov body $\oklim_{Y_\bullet}(D)$ in $\{0\}^{n-\kappa(D)}\times\R^{\kappa(D)}$ as a subset of $\R^{\kappa(D)}$ (see Theorem \ref{chpwmain}).

\medskip

Now, for a sufficiently large integer $m>0$, we take a log resolution $f_m \colon X_m \to X$ of the base ideal $\frak{b}(\lfloor mD \rfloor)$ so that we obtain a decomposition $f_m^*(\lfloor mD \rfloor)=M_m' + F_m'$ into a base point free divisor $M_m'$ and the fixed part $F_m'$ of $|f_m^*(\lfloor mD \rfloor)|$. Let $M_m:=\frac{1}{m}M_m'$ and $F_m:=\frac{1}{m}F_m'$.
We may assume that the $f_m \colon X_m \to X$ is isomorphic over a neighborhood of $x$.
Let $f_m^*D=P_m+N_m$ be the divisorial Zariski decomposition. By Theorem \ref{abprop} (2), it is also the $s$-decomposition.

\medskip

Since $Y_n$ is general, by taking the strict transforms $Y_i^m$  of $Y_i$ on $X_m$, we obtain an admissible flag $Y_{\bullet}^m : Y_0^m \supseteq \cdots \supseteq Y_n^m $ on $X_m$.
We note that  $V_m:=Y_{n-\kappa(D)}^m$ is also a positive volume subvariety of $f^*_mD$ since $f_m$ is $V$-birational (\cite[Proposition 2.24]{CHPW1}). By definition, we also see that $V_m$ is also a Nakayama subvariety of $M_m$. Clearly, it is also a positive volume subvariety of $M_m$.

\medskip

The following lemma is obvious (cf. \cite[Lemma II.2.11]{nakayama}).

\begin{lemma}\label{pullbackrdiv}
Let $f \colon X \to Y$ be a surjective morphism with connected fibers between smooth projective varieties, and $D$ be an effective divisor on $Y$. Then $H^0(X, \lfloor f^*(mD) \rfloor) = H^0(Y, \lfloor mD \rfloor)$ for a sufficiently large integer $m>0$.
\end{lemma}

\begin{proof}
We can write $\lfloor f^*(mD) \rfloor =f^*\lfloor mD \rfloor + \lfloor f^*\{ mD\} \rfloor$. Note that for every irreducible component $E$ of $\Supp\lfloor f^*\{ mD\} \rfloor$, we have codim$f(E) \geq 2$. By the projection formula, we obtain $f_* \lfloor f^*(mD) \rfloor = \lfloor mD \rfloor$, and the assertion follows.
\end{proof}

We now prove a version of Fujita's approximation for an abundant divisor, which is a generalization of \cite[Proposition 3.7]{lehmann-nu}.

\begin{lemma}\label{fujita}
Under the same notations as above, for a sufficiently large integer $m>0$, there exists an ample divisor $H$ on $X$ such that
$$
M_m \leq P_m \leq M_m + \frac{1}{m}f_m^*H.
$$
\end{lemma}

\begin{proof}
By Theorem \ref{thrm-abundant map}, we can take a birational morphism $\mu \colon W \to X$ with $W$ smooth and a contraction $g \colon W \to T$ such that for some big divisor $B$ on $T$, we have $P' \sim_{\Q} P''$ where $\mu^*D = P'+N'$ and $g^*B = P''+N''$ are the divisorial Zariski decompositions. By taking further blow-ups of $T$, we may assume that $T$ is smooth.
For any sufficiently large integer $m >0$, as in \cite[Proof of Proposition 3.7]{lehmann-nu}, we consider a log resolution of $h_m \colon T_m \to T$ of the base ideal $\frak{b}(\lfloor mB \rfloor)$ and the asymptotic multiplier ideal $\mathcal{J}(||mB||)$ so that we obtain a decomposition $h_m^*(\lfloor mB \rfloor)=M_m''' + F_m'''$ into a base point free divisor $M_m'''$ and the fixed part $F_m'''$ of $|h_m^*(\lfloor mB \rfloor)|$. Let $M_m'' := \frac{1}{m}M_m'''$ and $F_m'':=\frac{1}{m}F_m'''$.
Now, for a sufficiently large integer $m>0$, we take a log resolution $f_m^W \colon X_m^W \to W$ of the base ideal $\frak{b}(\lfloor m\mu^*D \rfloor)$ so that we obtain a decomposition $(f_m^W)^*(\lfloor m\mu^*D \rfloor)={M_m^W}' + {F_m^W}'$ into a base point free divisor ${M_m^W} '$ and the fixed part ${F_m^W} '$ of $|(f_m^W)^*(\lfloor m\mu^*D \rfloor)|$. Let $M_m^W:=\frac{1}{m}{M_m^W }'$ and $F_m^W:=\frac{1}{m}{F_m^W} '$. Note that for a sufficiently large $m'>m$, we may take  birational morphisms $h_{m',m} \colon T_{m'}  \to T_m$ and $f_{m',m}^W \colon X_{m'}^W \to X_m^W$.
We can assume that there are contractions $g_m \colon X_m^W \to T_m$ for sufficiently large integers $m>0$. Thus we have the following commutative diagram:
\[
\begin{split}
\xymatrix{
X_{m'}^W \ar[r]^{f_{m',m}^W} \ar[d]_{g_{m'}} & X_m^W \ar[r]^{f_m^W}  \ar[d]_{g_m} & W \ar[r]^{\mu} \ar[d]_{g} & X\\
T_{m'} \ar[r]_{h_{m',m}} & T_m \ar[r]_{h_m} & T .&
}
\end{split}
\]

We now claim that
\begin{equation}\label{fujitaclaim}
M_m^W \sim_{\Q} g_m^* M_m''~\text{ for any sufficiently large and divisible integer $m > 0$}.
\end{equation}
 We can assume that $D$ itself is an effective divisor. By applying Lemma \ref{pullbackrdiv}, we obtain
 $$
\begin{array}{l}
H^0(X, \lfloor mD \rfloor) = H^0(W, \lfloor \mu^*(mD) \rfloor) = H^0(W, \lfloor mP' \rfloor)\\
= H^0(W, \lfloor mP'' \rfloor) = H^0(W, \lfloor g^*(mB) \rfloor) = H^0(T, \lfloor mB \rfloor).
\end{array}
$$
We then have
\begin{small}
$$
H^0(X_m^W, mM_m^W) =  H^0(X, \lfloor mD \rfloor) = H^0(T, \lfloor mB \rfloor)= H^0(T_m, mM_m'')=H^0(X_m^W, g_m^*(mM_m'')).
$$
\end{small}\\[-16pt]
Note that $M_m^W \leq (f_m^W)^*P' \sim_{\Q} (f_m^W)^*P'' \geq g_m^* M_m''$ and
\begin{small}
$$
H^0(X_m^W, mM_m^W) = H^0(X_m^W, \lfloor m(f_m^W)^*P'\rfloor ) = H^0(X_m^W, \lfloor m(f_m^W)^*P'' \rfloor) = H^0(X_m^W, g_m^*(mM_m'')).
$$
\end{small}\\[-16pt]
Since $mM_m^W, g_m^*(m M_m'')$ are base point free, we obtain $M_m^W \sim_{\Q} g_m^* M_m''$ as desired.

Let $h_m^*B = P_m' + N_m'$ be the divisorial Zariski decomposition.
By \cite[Proposition 3.7]{lehmann-nu}, there exists an effective divisor $E'$ on $T$ such that $M_m'' \leq P_m' \leq M_m'' + \frac{1}{m}{h'}_{m}^*E'$.
(Even though this assertion is slightly different from the actual statement of \cite[Proposition 3.7]{lehmann-nu}, Lehmann actually proved this assertion in its proof.)
Thus we have
$$
{h}^*_{m',m} M_m'' \leq M_{m'}'' \leq P_{m'}' \leq {h}^*_{m',m}P_m' \leq {h}^*_{m',m} M_m'' + \frac{1}{m} {h}_{m',m}^*{h}_m^* E'
= {h}^*_{m',m} M_m'' + \frac{1}{m} {h}_{m'}^* E'
$$
so that $0 \leq M_{m'}'' - {h}^*_{m',m} M_m'' \leq \frac{1}{m} {h}^*_{m'}E'$.
Let $E:= g^* E'$. By taking pullback via $g_{m'}$ and by applying the claim (\ref{fujitaclaim}), we obtain
$0 \leq M_{m'}^W - (f_{m',m}^W)^* M_m^W \leq \frac{1}{m} (f_{m'}^W)^* E$. By taking pushforward via $f_{m',m}^W$, we then have $0 \leq f_{m',m *}^W M_{m'}^W - M_m^W \leq \frac{1}{m} (f_m^W)^* E$.
Let $(f_m^W)^* \mu^* D = P_m^W + N_m^W$ be the divisorial Zariski decomposition, which is also the $s$-decomposition by Theorem \ref{abprop} (2).
By definition of $s$-decomposition, $\displaystyle P_m^W = \lim_{m' \to \infty} f_{m',m*}^W M_{m'}^W$. Hence we obtain $0 \leq P_m^W - M_m^W \leq \frac{1}{m}(f_m^W)^*E$.
We can take an ample divisor $H$ on $X$ such that $\mu^*H \geq E$. Then we have
\begin{equation}\label{fujitaeq}
M_m^W \leq P_m^W \leq M_m^W + \frac{1}{m}(f_m^W)^* \mu^* H.
\end{equation}

To finish the proof, consider a common log resolution $f_m' \colon Z \to X$ of $\mu \circ f_m^W \colon X_m^W \to X$ and the log resolution $f_m \colon X_m \to X$ of $\frak{b}(\lfloor mD \rfloor)$ with the morphisms $p \colon Z \to X_m^W$ and $q \colon Z \to X_m$. Note that $M_m^Z:=p^* M_m^W = q^* M_m$ is also a base point free divisor. Let $(f_m')^* D = P_m^Z + N_m^Z$ be the divisorial Zariski decomposition. It is clear that $M_m^Z \leq P_m^Z$. On the other hand, since $P_m^Z \leq p^* P_m^W$, it follows from (\ref{fujitaeq}) that $P_m^Z \leq M_m^Z + \frac{1}{m} (f_m')^* H$. Notice that $P_m = q_* P_m^Z$. Thus by taking pushforward via $q$, we finally obtain
$$
M_m \leq P_m \leq M_m + \frac{1}{m}f_m^*H.
$$
This completes the proof.
\end{proof}

\begin{remark}
When $D$ is a big divisor, one can easily deduce the classical statement of Fujita's approximation (see e.g., \cite[Theorem 11.4.4]{pos}) from Lemma \ref{fujita}.
\end{remark}

The following is the main result of this subsection. This generalizes \cite[Theorem D]{lm-nobody} to the limiting Okounkov body case.

\begin{lemma}\label{altconlim}
Under the same notations as above, we have
$$\oklim_{Y_\bullet}(D)=\lim_{m \to \infty} \okbd_{Y_{n-\nu_{\BDPP}(D)\bullet}^m}(M_m|_{V_m}).$$
\end{lemma}

\begin{proof}
We treat $\oklim_{Y_\bullet}(D)$ and $\okbd_{Y_{n-\nu_{\BDPP}(D)\bullet}^m}(M_m|_{V_m})$ in the statement as the subsets of the same fixed space $\R^{\nu_{\BDPP}(D)}$.
For any sufficiently large $m'>0$, by Lemmas \ref{okbdbir} and \ref{okbdzd}, we have
$$
\oklim_{Y_\bullet}(D)=\oklim_{Y_\bullet^{m'}}(f_{m'}^*D)=\oklim_{Y_\bullet^{m'}}(P_{m'}).
$$
Thus $\oklim_{Y_\bullet^{m'}}(P_{m'})$ is independent of ${m'}$.
By  \cite[Lemma 5.5]{CPW2}, for any $m>0$, we have
$$
\oklim_{Y_\bullet^m}(M_m)=\oklim_{Y_{n-\nu_{\BDPP}(D)\bullet}^m}(M_m|_{V_m}).
$$
To prove the lemma, it is sufficient to verify $\displaystyle \oklim_{Y_\bullet^{m'}}(P_{m'})=\lim_{m\to\infty}\oklim_{Y_\bullet^m}(M_m)$.

By Lemma \ref{fujita}, for any sufficiently large integer $m>0$, we have
$$
M_m \leq P_m \leq M_m + \frac{1}{m}f_m^*H
$$
for some ample divisor $H$ on $X$. Since $x \in X$ is general, we may assume $x \not\in \Supp(H)$.
By the subadditivity property of limiting Okounkov bodies,
$$
\oklim_{Y_\bullet^m}(P_m-M_m) + \oklim_{Y_\bullet^m}(\frac{1}{m}f_m^*H +M_m - P_m) \subseteq \oklim_{Y_\bullet^m}( \frac{1}{m}f_m^*H)=\frac{1}{m} \okbd_{Y_{\bullet}}(H).
$$
Since $\displaystyle \lim_{m \to \infty} \frac{1}{m} \okbd_{Y_{\bullet}}(H) = \{ 0 \}$, it follows that
$$
\lim_{m \to \infty} \oklim_{Y_\bullet^m}(P_m-M_m) =  \lim_{m \to \infty} \oklim_{Y_\bullet^m}\left( \frac{1}{m}f_m^*H +M_m - P_m \right) = \{ 0 \}.
$$
By the subadditivity property of limiting Okounkov bodies,
$$
\begin{array}{l}
\oklim_{Y_\bullet^m}(M_m)+\oklim_{Y_\bullet^m}(P_m - M_m)\subseteq\oklim_{Y_\bullet^m}(P_m)\\
\oklim_{Y_\bullet^m}(P_m) + \oklim_{Y_\bullet^m}( \frac{1}{m}f_m^*H + M_m -P_m) \subseteq \oklim_{Y_\bullet^m}(M_m + \frac{1}{m}f_m^*H).
\end{array}
$$
Since $\oklim_{Y_\bullet^m}(P_m) \subseteq \R^{\nu_{\BDPP}(D)}$ and $Y^m_{n-\nu_{\BDPP}(D)} \not\subseteq \bp(M_m + \frac{1}{m}f_m^*H)$, it follows from Lemma \ref{fujita} and \cite[Theorem 1.1]{CPW2} that
$$
\lim_{m \to \infty} \oklim_{Y_\bullet^m}(M_m) \subseteq \lim_{m \to \infty}\oklim_{Y_\bullet^m}(P_m) \subseteq \lim_{m \to \infty} \okbd_{Y^m_{n-\nu_{\BDPP}(D) \bullet}}\left( M_m + \frac{1}{m}f_m^*H \right).
$$
The existence of the limits is guaranteed by the following claim:
\begin{equation}\label{*}
\lim_{m \to \infty} \vol_{\R^{\nu_{\BDPP}(D)}}(\oklim_{Y_\bullet^m}(M_m)) = \lim_{m \to \infty} \vol_{\R^{\nu_{\BDPP}(D)}}\left( \okbd_{Y^m_{n-\nu_{\BDPP}(D) \bullet}}\left( M_m + \frac{1}{m}f_m^*H \right) \right).
\end{equation}
If this claim (\ref{*}) holds, then
$$
\lim_{m \to \infty} \oklim_{Y_\bullet^m}(M_m) = \lim_{m \to \infty}\oklim_{Y_\bullet^m}(P_m)=\lim_{m \to \infty} \okbd_{Y^m_{n-\nu_{\BDPP}(D) \bullet}}\left( M_m + \frac{1}{m}f_m^*H \right).
$$
As we saw in the beginning of the proof, $\oklim_{Y_\bullet^m}(P_m)$ coincide with $\oklim_{Y_\bullet}(D)$ for all sufficiently large $m>0$. Thus we have
$$
\oklim_{Y_\bullet}(D) = \lim_{m \to \infty}\oklim_{Y_\bullet^m}(M_m).
$$

It now remains to prove the claim (\ref{*}). We may assume that $V_m:=Y_{n-\nu_{\BDPP}(D)}^m$ is a smooth positive volume subvariety of $M_m$, and $f_m|_{V_m} \colon V_m \to V$ is a birational contraction. By \cite[Lemma 5.5]{CPW2}, we have
$$
\vol_{\R^{\nu_{\BDPP}(D)}}(\oklim_{Y_\bullet^m}(M_m)) = \frac{1}{\nu_{\BDPP}(D)!} \vol_{V_m}(M_m|_{V_m}) = \frac{1}{\nu_{\BDPP}(D)!} (M_m|_{V_m})^{\nu_{\BDPP}(D)}.
$$
Similarly, by \cite[(2.7) in p.804]{lm-nobody}, we also have
\begin{small}
 $$
 \begin{array}{l}
\vol_{\R^{\nu_{\BDPP}(D)}}\left( \okbd_{Y^m_{n-\nu_{\BDPP}(D) \bullet}}\left( M_m|_{V_m} + \frac{1}{m}(f_m^*H)|_{V_m} \right) \right)\\
=\frac{1}{\nu_{\BDPP}(D)!} \vol_{V_m}( M_m|_{V_m} + \frac{1}{m}(f_m^*H)_{V_m})\\
=\frac{1}{\nu_{\BDPP}(D)!}\left( M_m|_{V_m} + \frac{1}{m}(f_m^* H)|_{V_m} \right)^{\nu_{\BDPP}(D)}\\
=\frac{1}{\nu_{\BDPP}(D)!}  \left( (M_m|_{V_m})^{\nu_{\BDPP}(D)} + \sum_{k=0}^{\nu_{\BDPP}(D)-1}  \frac{{\nu_{\BDPP}(D) \choose k}}{m^{\nu_{\BDPP}(D)-k}}(M_m|_{V_m})^k \cdot ((f_m^* H)|_{V_m})^{\nu_{\BDPP}(D)-k} \right).
\end{array}
$$
\end{small}\\[-7pt]
To prove the claim  (\ref{*}), it is sufficient to show that for each $0 \leq k \leq \nu_{\BDPP}(D)-1$, there exists a constant $C_k$ independent of $m$ such that
$$
(M_m|_{V_m})^k \cdot ((f_m^* H)|_{V_m})^{\nu_{\BDPP}(D)-k}\leq C_k.
$$
If $k=0$, then we have $((f_m^* H)|_{V_m})^{\nu_{\BDPP}(D)} = ((f_m|_{V_m})^*(H|_{V}))^{\nu_{\BDPP}(D)}=(H|_{V})^{\nu_{\BDPP}(D)}$, which is independent of $m$. Now, suppose that $1 \leq k \leq \nu_{\BDPP}(D)-1$.
Note that $V_m \not\subseteq \SB(f_m^*D)$ and $M_m|_{V_m} \leq (f_m^*D)|_{V_m}$. Thus
$$
M_m|_{V_m} \cdot ((f_m^*H)|_{V_m})^{\nu_{\BDPP}(D)-1} \leq (f_m^*D)|_{V_m} \cdot ((f_m^*H)|_{V_m})^{\nu_{\BDPP}(D)-1} = D|_{V} \cdot (H|_{V})^{\nu_{\BDPP}(D)-1}.
$$
By a Hodge-type inequality \cite[Corollary 1.6.3 (i)]{pos}, we have
\begin{footnotesize}
 $$
(M_m|_{V_m})^k \cdot ((f_m^*H)|_{V_m})^{\nu_{\BDPP}(D)-k} \leq \frac{(M_m|_{V_m} \cdot ((f_m^*H)|_{V_m})^{\nu_{\BDPP}(D)-1})^k}{ ((f_m^* H)|_{V_m})^{\nu_{\BDPP}(D)})^{k-1}} \leq  \frac{(D|_{V} \cdot (H|_{V})^{\nu_{\BDPP}(D)-1})^k}{((H|_{V})^{\nu_{\BDPP}(D)})^{k-1}}.
$$
\end{footnotesize}\\[-13pt]
Note that the right hand side is independent of $m$. This proves the claim (\ref{*}). We complete the proof.
\end{proof}

\section{Numerical equivalence and Okounkov body}\label{jowsec}

In this section, we prove Theorem \ref{main1} as Corollary \ref{numeq=same}. Throughout the section, $X$ is a smooth projective variety of dimension $n$. First, we need the following lemma.

\begin{lemma}\label{valbimor}
Let $f \colon \widetilde{X} \to X$ be a birational morphism with $\widetilde{X}$ smooth, and $D$ be a divisor on $X$ with $\kappa(D) \geq 0$. Consider an admissible flag
$$
\widetilde{Y}_\bullet: \widetilde{X}=\widetilde{Y}_0 \supseteq \widetilde{Y}_1 \supseteq \cdots \supseteq \widetilde{Y}_{n-1} \supseteq \widetilde{Y}_n=\{ x' \}
$$
on $\widetilde{X}$ and an admissible flag
$$
Y_\bullet: X=Y_0 \supseteq Y_1 \supseteq \cdots \supseteq Y_{n-1} \supseteq Y_n=\{ x \}
$$
on $X$ such that each restriction $f|_{\widetilde{Y}_{i}} \colon \widetilde{Y}_{i} \to Y_{i}$ is a birational morphism for $0 \leq i \leq n$.
Assume that $Y_i$ and $\widetilde{Y}_i$ are smooth for $0 \leq i \leq n$.
For $1 \leq i \leq n$, write $f|_{\widetilde{Y}_{i-1}}^*Y_i = \widetilde{Y}_i + E_i$ for some effective $f|_{\widetilde{Y}_{i-1}}$-exceptional divisor $E_i$ on $\widetilde{Y}_{i-1}$.
Then we have
$$
\okval_{\widetilde{Y}_\bullet}(f^*D)=\left\{ \mathbf{x}+ \sum_{i=1}^{n-1} x_i \cdot \nu_{\widetilde{Y}_{i\bullet}}(E_i|_{\widetilde{Y}_i}) ~\left|\;~ \mathbf{x}=(x_1, \ldots, x_n) \in \okval_{Y_\bullet}(D)  \right.\right \}
$$
where we regard $\nu_{\widetilde{Y}_{i\bullet}}(E_i|_{\widetilde{Y}_i})$ as a point in $\{ 0 \}^i \times \R^{n-i} \subseteq \R^n$.
In particular, $\okval_{Y_\bullet}(D)$ and $\okval_{\widetilde{Y}_\bullet}(f^*D)$ determine each other.
\end{lemma}

\begin{proof}
We can canonically identify $|D|_{\R}$ with $|f^*D|_{\R}$. For any $D' \in |D|_{\R}$, let
$$
\nu_{Y_\bullet}(D')=(\nu_1, \ldots, \nu_n) ~~\text{ and } ~~ \nu_{\widetilde{Y}_\bullet}(f^*D')=(\widetilde{\nu}_1, \ldots, \widetilde{\nu}_n).
$$
Since $\nu_{Y_\bullet}(|D|_{\R}|$ and $\nu_{\widetilde{Y}_\bullet}(|f^*D|_{\R})$ are dense subsets of $\okval_{Y_\bullet}(D)$ and $\okval_{\widetilde{Y}_\bullet}(f^*D)$, respectively, it is sufficient to show that
\begin{equation}\label{valbimoreq}
(\widetilde{\nu}_1,   \ldots, \widetilde{\nu}_n) = (\nu_1, \ldots, \nu_n) + \sum_{i=1}^{n-1} \nu_i \cdot \nu_{\widetilde{Y}_{i\bullet}}(E_i|_{\widetilde{Y}_i}).
\end{equation}

Let $D_1':=D'$ on $X=Y_0$, and define inductively $D_i':=(D_{i-1}'-\nu_{i-1}Y_{i-1})|_{Y_{i-1}}$ on $Y_{i-1}$  for $2 \leq i \leq n$. Similarly, let $\widetilde{D}_1':=f^*D'$ on $\widetilde{X}=\widetilde{Y}_0$, and define inductively $\widetilde{D}_i':=(\widetilde{D}_{i-1}'-\widetilde{\nu}_{i-1}\widetilde{Y}_{i-1})|_{\widetilde{Y}_{i-1}}$ on $\widetilde{Y}_{i-1}$ for $2 \leq i \leq n$. Then $\nu_i = \ord_{Y_i} D_i'$ and $\widetilde{\nu}_i = \ord_{\widetilde{Y}_i} \widetilde{D}_i'$ for $1 \leq i \leq n$.
First of all, observe that the first coordinates of both sides in (\ref{valbimoreq}) are $\widetilde{\nu}_1$ and $\nu_1$ and $\widetilde{\nu}_1=\nu_1$.
As $\widetilde{Y}_1 = f^* Y_1 - E_1$, we get
$$
\widetilde{D}_2'=(f^*D_1' - \nu_1 \widetilde{Y}_1)|_{\widetilde{Y}_1} = (f^*(D_1' - \nu_1 Y_1) + \nu_1 E_1)|_{\widetilde{Y}_1}
= f|_{\widetilde{Y}_1}^*D_2' + \nu_1 E_1|_{\widetilde{Y}_1}.
$$
Then we have
$$
\nu_{\widetilde{Y}_{1\bullet}}(\widetilde{D}_2') = \nu_{\widetilde{Y}_{1\bullet}}(f|_{\widetilde{Y}_1}^*D_2' ) + \nu_1 \cdot \nu_{\widetilde{Y}_{1\bullet}}(E_1|_{\widetilde{Y}_1}).
$$
Note that $\ord_{\widetilde{Y}_2} \widetilde{D}_2'=\widetilde{\nu}_2$ and $\ord_{\widetilde{Y}_2} f|_{\widetilde{Y}_1}^*D_2' = \nu_2$. Thus (\ref{valbimoreq}) holds for the second coordinates. Now, as $\widetilde{Y}_2 = f|_{\widetilde{Y}_1}^* Y_2 - E_2$, we get
$$
(f|_{\widetilde{Y}_1}^*D_2' - \nu_{2} \widetilde{Y}_2 )|_{\widetilde{Y}_2} = (f|_{\widetilde{Y}_1}^* (D_2' - \nu_2 Y_2) + \nu_2 E_2)|_{\widetilde{Y}_2}
= f|_{\widetilde{Y}_2}^*D_3' + \nu_2 E_2|_{\widetilde{Y}_2},
$$
Then we have
$$\nu_{\widetilde{Y}_{2\bullet}}((f|_{\widetilde{Y}_1}^*D_2' - \nu_{2} \widetilde{Y}_2 )|_{\widetilde{Y}_2} ) = \nu_{\widetilde{Y}_{2\bullet}}(f|_{\widetilde{Y}_2}^*D_3') + \nu_2 \cdot \nu_{\widetilde{Y}_{2\bullet}}(E_2|_{\widetilde{Y}_2}).$$
Note that
$$
\text{$\ord_{\widetilde{Y}_3} (f|_{\widetilde{Y}_1}^*D_2' - \nu_{2} \widetilde{Y}_2 )|_{\widetilde{Y}_2} +$the third coordinate of $\nu_1 \cdot \nu_{\widetilde{Y}_{1\bullet}}(E_1|_{\widetilde{Y}_1})=\widetilde{\nu}_3$}
$$
and $\ord_{\widetilde{Y}_3} f|_{\widetilde{Y}_2}^*D_3' = \nu_3$. Thus (\ref{valbimoreq}) holds for the third coordinates.
In general, we have
$$
\nu_{\widetilde{Y}_{i\bullet}}((f|_{\widetilde{Y}_{i-1}}^*D_i' - \nu_{i} \widetilde{Y}_i )|_{\widetilde{Y}_i} ) = \nu_{\widetilde{Y}_{i\bullet}}(f|_{\widetilde{Y}_i}^*D_{i+1}') + \nu_i \cdot \nu_{\widetilde{Y}_{i\bullet}}(E_i|_{\widetilde{Y}_i})~~\text{for $2 \leq i \leq n-1$. }
$$
Note that
$$
\text{$\ord_{\widetilde{Y}_{i+1}} (f|_{\widetilde{Y}_{i-1}}^*D_i' - \nu_{i} \widetilde{Y}_i )|_{\widetilde{Y}_i}
+$the $(i+1)$-th coordinate of $\displaystyle \sum_{j=1}^{i-1} \nu_{j} \cdot \nu_{\widetilde{Y}_{j\bullet}}(E_j|_{\widetilde{Y}_j})
=\widetilde{\nu}_{i+1}$}
$$
and
$\ord_{\widetilde{Y}_{i+1}}f|_{\widetilde{Y}_{i}}^* D_{i+1}' = \nu_{i+1}$.
Thus we obtain (\ref{valbimoreq}).
\end{proof}

We first prove the `only if' direction of Theorem \ref{main1}, which is a generalization of \cite[Proposition 4.1 (i)]{lm-nobody} to (possibly non-big) abundant divisors.

\begin{proposition}\label{numeq=>same}
Let $D, D'$ be divisors on $X$ with $\kappa(D), \kappa(D') \geq 0$. Suppose that $D$ or $D'$ is an abundant divisor. If $D \equiv D'$, then $\okval_{Y_\bullet}(D)=\okval_{Y_\bullet}(D')$ for every admissible flag $Y_\bullet$ on $X$.
\end{proposition}

\begin{proof}
By Theorem \ref{abprop} (1), both $D, D'$ are abundant divisors.
Fix an admissible flag $Y_\bullet$ on $X$. Possibly by taking a higher birational model of $X$, we may assume that each subvariety $Y_i$ in $Y_\bullet$ is smooth (see Remark \ref{smflag}).
By Theorem \ref{thrm-abundant map}, there is a birational morphism $\mu \colon W \to X$ and a morphism $g \colon W \to T$ with connected fibers such that $P_{\sigma}(\mu^*D) \sim_{\Q} P_{\sigma}(g^*B)$ for some big divisor $B$ on $T$.
Thus $P_{\sigma}(\mu^*D')|_F \equiv 0$ for a general fiber $F$ of $g$, and hence, $P_{\sigma}(\mu^*D')|_F \sim_{\Q} 0$ since $\kappa(\mu^*D')=\kappa(D') \geq 0$. This implies that $\kappa_{\sigma}(P_{\sigma}(\mu^*D')|_F)=\kappa(P_{\sigma}(\mu^*D')|_F)=0$. By taking a higher birational model of $W$ if necessary, by \cite[Corollary V.2.26]{nakayama} (see also \cite[Theorem 5.7]{lehmann-red}),
we may assume that $P_{\sigma}(\mu^*D') \sim_{\Q} P_{\sigma}(g^*B')$ for some divisor $B'$ on $T$.
Applying \cite[Lemma III.5.15]{nakayama} (see also \cite[Proof of Corollary 6.3]{lehmann-red}), we see that $P_{\sigma}(B) \equiv P_{\sigma}(B')$ and $B'$ is also a big divisor on $T$. We also have $P_{\sigma}(\mu^*D) \sim_{\Q} P_{\sigma}(g^*B)=P_{\sigma}(g^*P_{\sigma}(B))$ and $P_{\sigma}(\mu^*D') \sim_{\Q} P_{\sigma}(g^*B')=P_{\sigma}(g^*P_{\sigma}(B'))$. We write $P_{\sigma}(B)= P_{\sigma}(B') + N$ for some numerically trivial divisor $N$ on $T$.
Then we have
    $$
 P_{\sigma}(\mu^*D)  \sim_{\Q} P_{\sigma}(g^*P_{\sigma}(B)) = P_{\sigma}(g^*P_{\sigma}(B')) + g^*N  \sim_{\Q} P_{\sigma}(\mu^*D') + g^*N.
$$

By successively taking strict transforms $\widetilde{Y}_i$ of $Y_i$ under the birational morphisms $\mu|_{\widetilde{Y}_{i-1}} \colon \widetilde{Y}_{i-1} \to Y_{i-1}$ for $1 \leq i \leq n$, we obtain an admissible flag
$$
\widetilde{Y}_\bullet: W=\widetilde{Y}_0 \supseteq \widetilde{Y}_1 \supseteq \cdots \supseteq \widetilde{Y}_{n-1} \supseteq \widetilde{Y}_n
$$
on $W$. Possibly by taking a higher birational model of $W$, we may assume that each subvariety of $\widetilde{Y}_\bullet$ is smooth. By Theorem \ref{abprop} (2), we have $P_{\sigma}(\mu^*D)=P_s(\mu^*D)$ and $P_{\sigma}(\mu^*D')=P_s(\mu^*D')$.
By Lemmas \ref{okbdzd} and \ref{valbimor}, it is sufficient to show that
$$
\okval_{\widetilde{Y}_\bullet}(P_{\sigma}(\mu^*D)) = \okval_{\widetilde{Y}_\bullet}(P_\sigma(\mu^*D')).
$$

Now, take an ample divisor $A$ on $T$ so that $A+kN$ is also an ample divisor for every $k \in \Z$. Choose a large integer $a>0$ such that $aP_{\sigma}(B')-A \sim_{\Q} E'$ for some effective divisor $E'$ on $T$.
Then $aP_{\sigma}(g^*P_{\sigma}(B')) - g^*A \sim_{\Q} E$ for some effective divisor $E$ on $W$.
For any integer $m>0$, we have
$$
(m+a) P_{\sigma}(\mu^*D)  \sim_{\Q} (m+a)(P_{\sigma}(\mu^*D')+g^*N) \sim_{\Q} mP_{\sigma}(\mu^*D') + E + g^*(A + (m+a)N).
$$
By the subadditivity property of the valuative Okounkov bodies, we have
$$
\okval_{\widetilde{Y}_\bullet}( P_{\sigma}(\mu^*D) ) \supseteq \frac{m}{m+a}\okval_{\widetilde{Y}_\bullet}(P_{\sigma}(\mu^*D')) + \frac{1}{m+a}\okval_{\widetilde{Y}_\bullet}(E) + \frac{1}{m+a}\okval_{\widetilde{Y}_\bullet}((g^*(A+(m+a)N))).
$$
Note that $g^*(A+(m+a)N)$ is a semiample divisor on $W$. Then we can find an effective divisor $E'' \in |g^*(A+(m+a)N)|_{\R}$ such that $\ord_{\widetilde{Y}_i}(E'') = 0$ for $1 \leq i \leq n$, so the origin of $\R^n$ is contained in $\okval_{\widetilde{Y}_\bullet}(g^*(A+(m+a)N))$. Hence we obtain
$$
\okval_{\widetilde{Y}_\bullet}( P_{\sigma}(\mu^*D) ) \supseteq \frac{m}{m+a}\okval_{\widetilde{Y}_\bullet}(P_{\sigma}(\mu^*D')) + \frac{1}{m+a}\okval_{\widetilde{Y}_\bullet}(E).
$$
By letting $m \to \infty$, we see that
$$
\okval_{\widetilde{Y}_\bullet}( P_{\sigma}(\mu^*D) ) \supseteq \okval_{\widetilde{Y}_\bullet}(P_{\sigma}(\mu^*D')).
$$
Similarly by replacing $D$ by $D'$ and $N$ by $-N$, we can also obtain the reverse inclusion. Therefore we complete the proof.
\end{proof}

\begin{remark}
Obviously, Proposition \ref{numeq=>same} does not hold without the assumption that $D$ or $D'$ is an abundant divisor (see \cite[Remark 3.13]{CHPW1}).
\end{remark}

For the converse of Proposition \ref{numeq=>same}, we need several lemmas.

\begin{lemma}\label{twofib}
Consider two surjective morphisms $f_1 \colon X \to Z_1$ and $f_2 \colon X \to Z_2$ with connected fibers. Suppose that $\dim Z_1 = \dim Z_2>0$ and $f_1, f_2$ are not birationally equivalent. Then for a general member $G \in |H|$ where $H$ is a very ample divisor on $Z_1$, the inverse image $f_1^{-1}(G)$ dominates $Z_2$ via $f_2$, i.e., we have $f_2(f_1^{-1}(G))=Z_2$.
\end{lemma}

\begin{proof}
Notice that $|f_1^*H|$ is a base point linear system.
Thus we may assume that $f_1^{-1}(G)=f_1^*G \in |f_1^*H|$ is a general member so that $f_1^{-1}(G)$ is a prime divisor on $X$. Suppose that $f_2(f_1^{-1}(G))$ does not dominate $Z_2$ via $f_2$. Then $f_2(f_1^{-1}(G))$ is contained in a prime divisor $D$ on $Z_2$. We then have $f_1^*G \leq f_2^* D$, so
$$
H^0(X, \mathcal{O}_X(mf_1^*G)) \subseteq H^0(X, \mathcal{O}_X(mf_2^*D))~~\text{ for any integer $m>0$.}
$$
In particular, $D$ is a big divisor on $Z_2$. Consider a rational map $\varphi \colon X \dashrightarrow Z'$ given by $|mf_2^*D|$ for a sufficiently large and divisible integer $m>0$. The rational map $\varphi$ factors through $Z_2$ via a birational map $\varphi' \colon Z_2 \dashrightarrow Z'$ given by $|mD|$, and $\varphi$ and $f_2$ are birationally equivalent.
Since $H^0(X, \mathcal{O}_X(mf_1^*G)) \subseteq H^0(X, \mathcal{O}_X(mf_2^*D))$, there is a rational map $\pi \colon Z' \dashrightarrow Z_1$. Note that $\pi \circ \varphi \colon X \dashrightarrow Z_1$ is birationally equivalent to $f_1 \colon X \to Z_1$. As $f_1$ has connected fibers, $\pi$ is birational and so is $\pi \circ \varphi' \colon Z_2 \dashrightarrow Z_1$. This implies that $f_1, f_2$ are birationally equivalent, so we get a contradiction.
\end{proof}

\begin{theorem}\label{sameiitaka}
Let $D, D'$ be divisors on $X$ with $\kappa(D), \kappa(D') > 0$. If $\okval_{Y_\bullet}(D)=\okval_{Y_\bullet}(D')$ for every admissible flag $Y_\bullet$ on $X$, then the Iitaka fibrations of $D, D'$ are all birationally equivalent.
\end{theorem}

\begin{proof}
Let $f \colon X' \to X$ be a birational morphism with the Iitaka fibrations $\phi \colon X' \to Z$ of $D$ and $\phi' \colon X' \to Z'$ of $D'$.
Since $\dim \okval_{Y_\bullet}(D)=\kappa(D)$ for any admissible flag $Y_\bullet$, we have  $\kappa(D)=\kappa(D')$ so that $\dim Z=\dim Z'$.
To derive a contradiction, suppose that $\phi, \phi'$ are not birationally equivalent. By Lemma \ref{twofib}, for a general member $G \in |H|$ where $H$ is a very ample divisor on $Z$, the inverse image $\phi^{-1}(G)$ dominates $Z'$ via $\phi'$.
We can take a general subvariety $V' \subseteq \phi^{-1}(G)$ of dimension $\kappa(D')$ such that $f(V')$ is a Nakayama subvariety of $D'$.
By Theorem \ref{nakpvscrit}, $f(V')$ is also a Nakayama subvariety of $D$. However, $\phi(V') \subseteq \phi(\phi^{-1}(G))=G$, so $V'$ does not dominate $Z$ via $\phi$.
This is a contradiction, and we are done.
\end{proof}

The following lemma plays a crucial role in the proof of the converse of Proposition  \ref{numeq=>same}. It can be considered as a generalization of \cite[Corollary 3.3 and Theorem 3.4 (b)]{Jow} although our proof is completely different from Jow's proof in \cite{Jow}.

\begin{lemma}\label{curveintersection}
Let $D$ be a divisor on $X$ with $\kappa(D) > 0$, and $D=P_s + N_s$ be the $s$-decomposition. Consider an irreducible curve $C$ on $X$ obtained by a transversal complete intersection of general effective very ample divisors on $X$. We can choose an admissible flag $Y_\bullet : X=Y_0 \supseteq Y_1 \supseteq \cdots \supseteq Y_{n-1} \supseteq Y_n=\{x\}$ on $X$ such that $Y_{n-\kappa(D)}$ is a Nakayama subvariety of $D$, $Y_{n-1}=C$, and $x$ is a general point on $C$.
Fix an Iitaka fibration $\phi \colon X' \to Z$ of $D$, and let $C'$ be the strict transform of $C$ on $X'$.
Then we have
$$
P_s \cdot C =\deg(C' \to \phi(C')) \cdot \vol_{\R^1}(\okval_{Y_\bullet}(D)_{x_1=\cdots=x_{n-1}=0}).
$$
\end{lemma}

\begin{proof}
We can choose general effective very ample divisors $A_1, \ldots, A_{n-1}$ on $X$ such that $A_1 \cap \cdots \cap A_{n-1}=C$. We may assume that $Y_i:=A_1 \cap \cdots \cap A_{i}$ is an irreducible subvariety of codimension $i$ for each $1 \leq i \leq n-1$. By letting $Y_n:=\{x\}$ where $x$ is a general point on $C$, we obtain an admissible flag
$$
Y_\bullet : X=Y_0 \supseteq Y_1 \supseteq \cdots \supseteq Y_{n-1} \supseteq Y_n=\{x\}
$$
on $X$. Since $A_1, \ldots, A_{n-\kappa(D)}$ are general effective very ample divisors, $Y_{n-\kappa(D)}$ is a Nakayama subvariety of $D$ by \cite[Proposition 2.13]{CHPW1}. Thus this admissible flag $Y_\bullet$ satisfies the conditions in the statement.

For a sufficiently large integer $m>0$, take a log resolution $f_m \colon X_m \to X$ of the base ideal $\mathfrak{b}(\lfloor mD \rfloor)$ so that we obtain a decomposition $f_m^*(\lfloor mD \rfloor)=M_m'+F_m'$ into a base point free divisor $M_m'$ on $X_m$ and the fixed part $F_m'$ of $|f_m^*(\lfloor mD \rfloor)|$. Let $M_m:=\frac{1}{m}M_m'$ and $F_m:=\frac{1}{m}F_m'$. Let $\phi_m \colon X_m \to Z_m$ be a morphism given by $|M_m'|$. By taking a higher birational model of $Z_m$, we may assume that $Z_m$ is a smooth variety.
There exists a nef and big divisor $H_m$ on $Z_m$ such that $M_m=\phi_m^* H_m$.
Since our choice of admissible flag $Y_\bullet$ is independent of this process, we may assume that $f_m \colon X_m \to X$ is isomorphic over a neighborhood of $x$. Let $C_m$ be the strict transform of $C$ on $X_m$.
By taking strict transforms $Y_i^m$ of $Y_i$ on $X_m$ for each $0 \leq i \leq n-1$ (note that $Y_{n-1}^m=C_m$), we obtain an admissible flag
$$
Y_{\bullet}^m: X_m=Y_0^m \supseteq Y_1^m \supseteq \cdots \supseteq Y_{n-1}^m \supseteq Y_n^m=\{f_m^{-1}(x)\}
$$
on $X_m$. We may also assume that
\begin{scriptsize}
$$
\overline{Y}^m_{\bullet} : Z_m=\overline{Y}^m_0=\phi_m(Y_{n-\kappa(D)}^m) \supseteq \overline{Y}^m_1=\phi_m(Y_{n-\kappa(D)-1}^m) \supseteq \cdots \supseteq \overline{Y}^m_{\kappa(D)-1}=\phi_m(Y_{n-1}^m) \supseteq \overline{Y}^m_{\kappa(D)}=\phi_m(Y_n^m)
$$
\end{scriptsize}\\[-18pt]
is an admissible flag on $Z_m$. Note that $d:=\deg(C' \to \phi(C')) = \deg(C_m \to \phi_m(C_m))$. Then $M_m \cdot C_m = d \cdot (H_m \cdot \phi_m(C_m))$. By \cite[Theorem 1.1]{CPW2}, we have
$$
H_m \cdot \phi_m(C_m) = \vol_{Z_m|\overline{Y}_{\kappa(D)-1}}(H_m)=\vol_{\R^1}(\okbd_{\overline{Y}^m_{\bullet}}(H_m)_{x_1=\cdots=x_{\kappa(D)-1}=0}).
$$

We now prove that $\displaystyle P_s \cdot C = \lim_{m \to \infty}M_m \cdot C_m$. Let $E_1, \ldots, E_k$ be the divisorial components of $\SB(D)$. Since the closure of $\SB(D) \setminus (E_1 \cup \cdots \cup E_k)$ has codimension at least two in $X$, we may assume that $C \cap \SB(D) \subseteq E_1 \cup \cdots \cup E_k$. We can also assume that $C$ is smooth and meets all $E_i$ transversally at smooth points of $E_i$. Thus $C_m$ does not meet any effective $f_m$-exceptional divisor. We write
$$
f_m^* \frac{\lfloor mP_s \rfloor}{m} = M_m + e_1^m f_{m*}^{-1}E_1 + \cdots+ e_k^m f_{m*}^{-1}E_k + F_m''
$$
where $F_m''$ is an effective $f_m$-exceptional divisor. We have
$$
\frac{\lfloor mP_s \rfloor}{m} \cdot C = f_m^*\frac{\lfloor mP_s \rfloor}{m} \cdot C_m = M_m \cdot C_m + e_1^m E_1 \cdot C + \cdots + e_k^m E_k \cdot C.
$$
Since $\displaystyle \lim_{m \to \infty} e_i^m=0$ for each $1 \leq i \leq k$ and $\displaystyle  \lim_{m \to \infty}\frac{\lfloor mP_s \rfloor}{m} \cdot C =P_s \cdot C$, it follows that $\displaystyle P_s \cdot C = \lim_{m \to \infty}M_m \cdot C_m$ as desired.

Combining what we have obtained above, we find
$$
P_s \cdot C = d \cdot \lim_{m \to \infty} \vol_{\R^1}(\okbd_{\overline{Y}^m_{\bullet}}(H_m)_{x_1=\cdots=x_{\kappa(D)-1}}).
$$
To prove the lemma, it is sufficient to show that
\begin{equation}\label{sharp}
\lim_{m \to \infty} \okbd_{\overline{Y}^m_{\bullet}}(H_m)_{x_1=\cdots=x_{\kappa(D)-1}=0}=\okval_{Y_\bullet}(D)_{x_1=\cdots=x_{n-1}=0}.
\end{equation}
By definition, $\displaystyle  \lim_{m \to \infty} \okbd_{\overline{Y}^m_{\bullet}}(H_m)_{x_1=\cdots=x_{\kappa(D)-1}=0}\subseteq \okval_{Y_\bullet}(D)_{x_1=\cdots=x_{n-1}=0}$ holds. To derive a contradiction, suppose that this inclusion is strict.
For a sufficiently large integer $m_0 >0$, we can choose a small ample $\Q$-divisor $A_{m_0}$ on $Z_{m_0}$ such that
$$
\vol_{\R^1}(\okbd_{\overline{Y}^{m}_{\bullet}}(H_{m})_{x_1=\cdots=x_{\kappa(D)-1}}) + A_{m_0}\cdot \phi_{m_0}(C_{m_0}) < \vol_{\R^1}(\okval_{Y_\bullet}(D)_{x_1=\cdots=x_{n-1}=0}) - \eps
$$
for any sufficiently small number $\eps > 0$ and any sufficiently large integer $m > m_0$.
There exists a sufficiently small number $\delta >0$ such that all the following divisors
$$
\begin{array}{rcl}
A_{m_0,1}=A_{m_0,1}(\delta_1)& \sim_{\Q} & A_{m_0}+\delta_1 \overline{Y}^{m_0}_1,\\
A_{m_0,2}=A_{m_0,2}(\delta_1, \delta_2)& \sim_{\Q} & A_{m_0,1}|_{\overline{Y}^{m_0}_1}+\delta_2 \overline{Y}^{m_0}_2,\\
  &\vdots&\\
A_{m_0, \kappa(D)-1}=A_{m_0, \kappa(D)-1}(\delta_1, \ldots, \delta_{\kappa(D)-1})& \sim_{\Q} & A_{m_0, \kappa(D)-2}|_{\overline{Y}_{\kappa(D)-2}^{m_0}} + \delta_{\kappa(D)-1} \overline{Y}^{m_0}_{\kappa(D)-1}
\end{array}
$$
are ample $\Q$-divisors for any nonnegative rational numbers $\delta_1, \delta_2, \ldots, \delta_{\kappa(D)-1} \leq \delta$.
By Lemma \ref{altconval},  $\displaystyle \okval_{Y_\bullet}(D) = \lim_{m \to \infty} \okval_{Y_{\bullet}^m}(M_m)= \lim_{m \to \infty} \okbd_{\overline{Y}_\bullet^m}(H_m)$.
Thus there exist a sufficiently large integer $m>0$ and an effective divisor $H_m' \sim_{\Q} H_m$ on $Z_m$ such that if we write $\nu_{\overline{Y}^m_{\bullet}}(H_m')=(\delta_1, \ldots, \delta_{\kappa(D)-1}, b)$, then $\delta_1, \delta_2, \ldots, \delta_{\kappa(D)-1}, b$ are nonnegative rational numbers with $\delta_1, \delta_2, \ldots, \delta_{\kappa(D)-1} \leq \delta$ and $
 \vol_{\R^1}(\okval_{Y_\bullet}(D)_{x_1=\cdots=x_{n-1}=0}) - \eps \leq b$.
We can write
 $$
\begin{array}{rcl}
H_m' & = & H_{m,1} + \delta_1 \overline{Y}_1^m,\\
H_{m,1}|_{\overline{Y}_1^m}& = & H_{m,2} + \delta_2 \overline{Y}_2^m,\\
  &\vdots&\\
H_{m, \kappa(D)-2}|_{\overline{Y}_{\kappa(D)-2}^m}&= & H_{m, \kappa(D)-1} + \delta_{\kappa(D)-1} \overline{Y}_{\kappa(D)-1}^m
\end{array}
$$
where each $H_{m,i}$ is an effective divisor on $\overline{Y}_{i-1}^m$ for $1 \leq i \leq \kappa(D)-1$. Notice that
$H_{m,\kappa(D)-1} \cdot \phi_m(C_m)=H_{m,\kappa(D)-1} \cdot \overline{Y}_{\kappa(D)-1}^m \geq b$.
By taking common resolution, we can assume that there is a birational morphism $g_{m} \colon Z_m \to Z_{m_0}$ such that $\overline{Y}_i^m = g_m|_{\overline{Y}_{i-1}^m}^* \overline{Y}_i^{m_0}$ for every $1 \leq i \leq \kappa(D)$.
Note that $H_m + g_m^* A_{m_0} + B$ is an ample $\Q$-divisor on $Z_m$ for any ample $\Q$-divisor $B$ on $Z_m$. We may assume that $\overline{Y}_{\kappa(D)-2}^m \not\subseteq \Supp(B)$. Thus we can find an effective divisor $E \sim_{\Q} H_m + g_m^* A_{m_0} + B$ such that
$$
E|_{\overline{Y}_{\kappa(D)-2}^m}=H_{m,\kappa(D)-1} + g_m|_{\overline{Y}_{\kappa(D)-2}^m}^*A_{m_0, \kappa(D)-1} + B|_{\overline{Y}_{\kappa(D)-2}^m}
$$
where $A_{m_0, \kappa(D)-1}=A_{m_0, \kappa(D)-1}(\delta_1, \ldots, \delta_{\kappa(D)-1})$. Then we obtain
\begin{scriptsize}
$$
(H_m + g_m^* A_{m_0} + B) \cdot \phi_m(C_m)
=E \cdot \phi_m(C_m)
=(H_{m,\kappa(D)-1} + g_m|_{\overline{Y}_{\kappa(D)-2}^m}^*A_{m_0,\kappa(D)-1} +B|_{\overline{Y}_{\kappa(D)-2}^m}) \cdot \phi_m(C_m) > b.
$$
\end{scriptsize}\\[-15pt]
As $B$ can be an arbitrarily small ample divisor, we get $(H_m + g_m^* A_{m_0}) \cdot \phi_m(C_m) \geq b$. Then we have
$$
\begin{array}{l}
  \vol_{\R^1}(\okbd_{\overline{Y}^{m}_{\bullet}}(H_{m})_{x_1=\cdots=x_{\kappa(D)-1}}) + A_{m_0}\cdot \phi_{m_0}(C_{m_0}) =  (H_m + g_m^* A_{m_0}) \cdot \phi_m(C_m)\\
 \geq b \geq \vol_{\R^1}(\okval_{Y_\bullet}(D)_{x_1=\cdots=x_{n-1}=0}) - \eps,
 \end{array}
$$
which is a contradiction. Therefore, we obtain (\ref{sharp}) as required.
\end{proof}

\begin{remark}
Here we explain why Lemma \ref{curveintersection} can be considered as a generalization of Jow's result \cite[Corollary 3.3 and Theorem 3.4 (b)]{Jow}, which states that if $D$ is a big divisor on $X$ and $Y_\bullet$ is an admissible flag on $X$ whose subvarieties are transversal complete intersections of general effective very ample divisors on $X$, then
$$
 \vol_{\R^1}(\okbd_{Y_\bullet}(D)_{x_1=\cdots=x_{n-1}=0}) = D \cdot Y_{n-1} - \sum_{i=1}^k \sum_{p \in Y_{n-1} \cap E_i} \ord_{E_i}(||D||)
$$
where $E_1, \ldots, E_k$ are irreducible components of $\SB(D)$. Since $Y_{n-1}$ is a sufficiently general curve, we may assume that $Y_{n-1}$ is smooth and meets all $E_i$ transversally at smooth points of $E_i$. Thus Jow's result can be also expressed equivalently as
$$ \sum_{i=1}^k \sum_{p \in Y_{n-1} \cap E_i} \ord_{E_i}(||D||) = N_{\sigma}(D) \cdot Y_{n-1}$$ so that $\vol_{\R^1}(\okbd_{Y_\bullet}(D)_{x_1=\cdots=x_{n-1}=0}) =P_{\sigma}(D) \cdot Y_{n-1}$. Note that for any big divisor $D$, $P_{\sigma}(D)=P_s(D)$ and the identity map $\text{id}_X \colon X \to X$ is an Iitaka fibration of $D$. Thus Lemma \ref{curveintersection} recovers Jow's result.
\end{remark}

\begin{lemma}\label{negpartviaval}
Let $D$ be a divisor on $X$ with $\kappa(D) > 0$, and $D=P_s + N_s$ be the $s$-decomposition. Let $E$ be an irreducible component of $N_s$. Then we have
$$
\mult_E N_s= \inf\{ x_1 \mid (x_1, \ldots, x_n) \in \okval_{Y_\bullet}(D), \text{ $Y_\bullet$ is an admissible flag such that $Y_1=E$} \}.
$$
In particular, one can read off the negative part $N_s$ from the set
$$
\{ \okval_{Y_\bullet}(D) \mid Y_\bullet \text{ is an admissible flag on $X$} \}.
$$
\end{lemma}

\begin{proof}
By the definition of $s$-decomposition, we have
$$
 \inf\{ x_1 \mid (x_1, \ldots, x_n) \in \okval_{Y_\bullet}(P_s), \text{ $Y_\bullet$ is an admissible flag such that $Y_1=E$} \}=0.
$$
Note also that $\okval_{Y_\bullet}(N_s)$ consists of a single point $(x_1, \ldots, x_n)$ with $x_1=\mult_E N_s$. Thus the assertion  follows from Lemma \ref{okbdzd}.
\end{proof}

We are now ready to complete the proof of Theorem \ref{main1} by proving the converse of Proposition \ref{numeq=>same}.
The following result is a generalization of \cite[Theorem A]{Jow} to possibly non-big divisor case.

\begin{proposition}\label{same=>numeq}
Let $D, D'$ be divisors on $X$ with $\kappa(D), \kappa(D') \geq 0$. If $\okval_{Y_\bullet}(D)=\okval_{Y_\bullet}(D')$ for every admissible flag $Y_\bullet$ on $X$, then $D \equiv D'$.
\end{proposition}

\begin{proof}
Recall that if $D$ is a divisor with $\kappa(D) \geq 0$, then any $\kappa(D)$-dimensional general subvariety of $X$ is a Nakayama subvariety of $D$. Thus we can take an admissible flag $Y_\bullet$ containing the Nakayama subvarieties of $D,D'$ with general $Y_n$.  By the assumption, we can deduce from Theorem \ref{chpwmain} (1) that $\kappa(D)=\kappa(D')$.
The assertion is trivial when $\kappa(D)=\kappa(D')=0$. Thus, from now on, we assume that $\kappa(D), \kappa(D') > 0$.
By Theorem \ref{sameiitaka}, we may fix an Iitaka fibration $\phi:X'\to Z$ for both $D$ and $D'$.
Let $D=P_s+N_s$ and $D'=P_s'+N_s'$ be the $s$-decompositions. By Lemma \ref{negpartviaval}, we have $N_s=N_s'$. Thus it is sufficient to show that $P_s \equiv P_s'$. By applying \cite[Lemma 3.5]{Jow}, we can take irreducible curves $C_1, \ldots, C_{\rho}$ on $X$ obtained by transversal complete intersections of general effective very ample divisors on $X$ in such a way that they form a basis of $N_1(X)_{\Q}$.
As in Lemma \ref{curveintersection}, for each $1 \leq i \leq \rho$, we can choose an admissible flag
$$
Y_\bullet^i : X=Y_0^i \supseteq Y_1^i \supseteq \cdots \supseteq Y_{n-1}^i \supseteq Y_n^i=\{x^i \}
$$
on $X$ such that $Y_{n-\kappa(D)}^i$ is a Nakayama subvariety of $D$, $Y_{n-1}^i=C_i$, and $x^i$ is a very general point on $C_i$.
For each $1 \leq i \leq \rho$, let $C_i'$ be the strict transform of $C_i$ on $X'$. By Lemma \ref{curveintersection} and the assumption, we have
$$
\begin{array}{rll}
P_s \cdot C_i & =~~\deg( C_i' \to \phi(C_i')) \cdot \vol_{\R^1}(\okval_{Y_\bullet^i}(D)_{x_1=\cdots=x_{n-1}=0}) &\\
   &= ~~ \deg( C_i' \to \phi(C_i')) \cdot \vol_{\R^1}(\okval_{Y_\bullet^i}(D')_{x_1=\cdots=x_{n-1}=0}) &=~~P_s' \cdot C_i
\end{array}
$$
for every $1 \leq i \leq \rho$. Thus $P_s \equiv P_s'$, and this finishes the proof.
\end{proof}

\begin{remark}
In Proposition \ref{same=>numeq}, we do not assume that $D$ or $D'$ is an abundant divisor.
Clearly, Proposition \ref{same=>numeq} does not hold without the assumption that $\kappa(D), \kappa(D') \geq 0$. We have $\kappa(D), \kappa(D') = -\infty$ for any non-pseudoeffective divisors $D$ and $D'$. However, $\okval_{Y_\bullet}(D)=\okval_{Y_\bullet}(D')=\emptyset$ for every admissible flag $Y_\bullet$ on $X$.
\end{remark}

As a consequence of Propositions \ref{numeq=>same} and \ref{same=>numeq}, we obtain Theorem \ref{main1} as Corollary \ref{numeq=same}.

\begin{corollary}\label{numeq=same}
Let $D, D'$ be divisors on $X$ with $\kappa(D), \kappa(D') \geq 0$. If $D$ or $D'$ is an abundant divisor, then $D \equiv D'$ if and only if $\okval_{Y_\bullet}(D)=\okval_{Y_\bullet}(D')$ for every admissible flag $Y_\bullet$ on $X$.
\end{corollary}

\begin{proof}
The assertion follows from Propositions \ref{numeq=>same} and \ref{same=>numeq}.
\end{proof}

Finally, we prove the following.

\begin{corollary}\label{lineq=same}
Let $D, D'$ be divisors on $X$ with $\kappa(D), \kappa(D') \geq 0$. If $\Pic(X)$ is finitely generated, then $D \sim_{\R} D'$ if and only if $\okval_{Y_\bullet}(D)=\okval_{Y_\bullet}(D')$ for every admissible flag $Y_\bullet$ on $X$.
\end{corollary}

\begin{proof}
The `only if' direction is trivial by definition (see also \cite[Proposition 3.13]{CHPW1}). For the converse, note that $D \equiv D'$ if and only if $D \sim_{\R} D'$ under the assumption that $\Pic(X)$ is finitely generated. Then the `if' direction follows from Proposition \ref{same=>numeq}.
\end{proof}

\section{Restricted base locus via Okounkov bodies}\label{b-sec}

We show Theorem \ref{main2} as Theorem \ref{b-thm} in this section. The idea of the proof is essentially the same as that of \cite[Theorem A]{CHPW2}, but we include the detailed proof for reader's convenience. Throughout the section, $X$ is a smooth projective variety of dimension $n$.

\begin{theorem}\label{b-thm}
Let $D$ be a pseudoeffective abundant divisor on $X$, and $x \in X$ be a point. Then the following are equivalent:
\begin{enumerate} [leftmargin=0cm,itemindent=.6cm]
 \item[$(1)$] $x \in \bm(D)$
 \item[$(2)$] $\okval_{Y_\bullet}(D)$ does not contain the origin of $\R^n$ for every admissible flag $Y_\bullet$ on $X$ centered at $x$.
  \item[$(3)$] $\okval_{Y_\bullet}(D)$ does not contain the origin of $\R^n$ for some admissible flag $Y_\bullet$ on $X$ centered at $x$.
\end{enumerate}
\end{theorem}

\begin{proof}
We may assume that $D$ is effective. Since $D$ is an abundant divisor, we have $\ord_{V}(||D||)=\inf\{ \sigma(D') \mid D \sim_{\R} D' \geq 0 \}$ by Theorem \ref{abprop} (2) for any  divisorial valuation $\sigma$ with the center $V$ on $X$.

\smallskip

\noindent $(1) \Rightarrow (2)$ Assume that $x \in \bm(D)$, and fix an admissible flag $Y_\bullet$ centered at $x$.
By taking a sufficiently small ample divisor $A$, we may assume that $x\in\bm(D+A)$.
By \cite[Theorem B]{elmnp-asymptotic inv of base}, we have  $\ord_x(||D+A||) > 0$. Thus it follows that
$$
\delta:=\inf\{ \mult_x(D') \mid D \sim_{\R} D' \geq 0 \} = \ord_x(||D||) \geq \ord_x(||D+A||) > 0.
$$
For $D' \in |D|_{\R}$, we write $\nu_{Y_\bullet}(D')=(\nu_1(D'), \ldots, \nu_n(D'))$. Then we obtain
$$
\nu_1(D') + \cdots + \nu_n(D') \geq \mult_x(D') \geq \delta.
$$
This implies that for any point $\mathbf{x}=(x_1, \ldots, x_n) \in \okval_{Y_\bullet}(D)$, we have $x_1 + \cdots + x_n \geq \delta$. In particular, $\okval_{Y_\bullet}(D)$ does not contain the origin of $\R^n$.

\smallskip

\noindent $(2) \Rightarrow (3)$ Trivial.

\smallskip

\noindent$(3) \Rightarrow (1)$ Assume that $x \not\in \bm(D)$, and fix an arbitrary admissible flag $Y_\bullet$ centered at $x$. By Remark \ref{smflag}, we may assume that each $Y_i$ in $Y_\bullet$ is smooth. We use the notations in Subsection \ref{altsubsec2}. We may take a birational morphism $f_m \colon X_m \to X$ for each sufficiently large integer $m >0$ in such a way that there is an admissible flag $Y_\bullet^m$ on $X_m$ such that $f_m|_{Y_i^m} \colon Y_i^m \to Y_i$ is a birational morphism for $0 \leq i \leq n$.
We can write $f_m^*D=M_m + N_m + (P_m-M_m)$. Note that $Y_n^m \not\subseteq \Supp(N_m)$ and $M_m$ is semiample. Thus there is an effective divisor $D'_m \sim_{\R} D$ such that
$$
\nu_{Y_\bullet^m}(D'_m) = \nu_{Y_\bullet^m}(P_m-M_m).
$$
Now, by Lemma \ref{fujita}, there is an ample divisor $H$ on $X$ such that $P_m - M_m \leq \frac{1}{m}f_m^*H$. We then have
$$
\nu_{Y_\bullet^m}(D'_m) + \mathbf{x} = \nu_{Y_\bullet^m}(P_m-M_m) + \mathbf{x} \in \okval_{Y_\bullet^m}\left(\frac{1}{m}f_m^*H \right)~\text{for some $\mathbf{x} \in \R^n_{\geq 0}$}.
$$
In view of Lemma \ref{valbimor}, we see that
$$
\nu_{Y_\bullet}(D'_m) + \mathbf{x}' \in \okval_{Y_\bullet}\left(\frac{1}{m}H \right)~\text{for some $\mathbf{x}' \in \R^n_{\geq 0}$}.
$$
However, since $\displaystyle \lim_{m \to \infty}\okval_{Y_\bullet}\left(\frac{1}{m}H \right) = \{ 0 \}$, it follows that $\displaystyle \lim_{m \to \infty}\nu_{Y_\bullet}(D'_m)=0$. This means that the origin of $\R^n$ is contained in $\okval_{Y_\bullet}(D)$. We have shown that $(3) \Rightarrow (1)$.
\end{proof}

\begin{corollary}\label{nefcrit}
Let $D$ be an abundant divisor on $X$. Then the following are equivalent:
\begin{enumerate} [leftmargin=0cm,itemindent=.6cm]
 \item[$(1)$] $D$ is nef.
 \item[$(2)$] For every point $x \in X$, there exists an admissible flag $Y_\bullet$ on $X$ centered at $x$ such that $\okval_{Y_\bullet}(D)$ contain the origin of $\R^n$.
  \item[$(3)$] $\okval_{Y_\bullet}(D)$ contain the origin of $\R^n$ for every admissible flag $Y_\bullet$ on $X$.
\end{enumerate}
\end{corollary}

\begin{proof}
Recall that a divisor $D$ on $X$ is nef if and only if $\bm(D)=\emptyset$. Thus the corollary is immediate from Theorem \ref{b-thm}.
\end{proof}

\begin{remark}\label{b-rmk}
Note that Theorem \ref{b-thm} and Corollary \ref{nefcrit} may not hold when $D$ is not abundant (see \cite[Remark 4.10]{CPW1}). The main reason is that for a divisorial valuation $\sigma$ with the center $V$ on $X$, we may have $\ord_V(||D||) \neq \inf\{ \sigma(D') \mid D \sim_{\R} D' \geq 0 \}$ in contrast to the abundant divisor case (Theorem \ref{abprop} (2)).
\end{remark}

\section{Comparing two Okounkov bodies}\label{compsec}

In this section, we prove Theorem \ref{main3} as Theorem \ref{comvallim}.

\begin{theorem}\label{comvallim}
Let $D$ be a pseudoeffective abundant divisor on an $n$-dimensional smooth projective variety $X$ with $\kappa(D)>0$. Fix an admissible flag $Y_\bullet$ on $X$ such that $V:=Y_{n-\kappa(D)}$ is a Nakayama subvariety of $D$ and $Y_n$ is a general point in $X$.
Consider the Iitaka fibration $\phi \colon X' \to Z$ of $D$ and the strict transform $V'$ of $V$ on $X'$. Then we have
$$
\vol_{\R^{\kappa(D)}}(\oklim_{Y_\bullet}(D))=\deg(\phi|_{V'} \colon V' \to Z) \cdot \vol_{\R^{\kappa(D)}}(\okval_{Y_\bullet}(D)).
$$
In particular, $\okval_{Y_\bullet}(D)=\oklim_{Y_\bullet}(D)$ if and only if the map $\phi|_{V'} \colon V' \to Z$ is generically injective.
\end{theorem}

\begin{proof}
We use the notations in Section \ref{altsec}. By Proposition \ref{nak=psv}, $V$ is also a positive volume subvariety of $D$.
For a sufficiently large integer $m>0$, we have
$$
\deg(\phi_m|_{V_m} \colon V_m \to Z_m)=\deg(\phi|_{V'} \colon V' \to Z)=:d.
$$
Since $\phi_m|_{V_m}^*H_m=M_m$, it follows that $\vol_{V_m}(M_m|_{V_m})=d\cdot\vol_{Z_m}(H_m)$. By Lemmas \ref{altconval}, \ref{altconlim}, Theorem \ref{chpwmain}, and \cite[Theorem A]{lm-nobody}, we obtain
\begin{footnotesize}
$$
\vol_{\R^{\kappa(D)}}(\oklim_{Y_\bullet}(D)) =\lim_{m \to \infty} \frac{1}{\kappa(D)!}\vol_{V_m}(M_m|_{V_m}) \text{ and }
\vol_{\R^{\kappa(D)}}(\okval_{Y_\bullet}(D)) = \lim_{m \to \infty} \frac{1}{\kappa(D)!}\vol_{Z_m}(H_m).
$$
\end{footnotesize}\\[-13pt]
Thus the first assertion immediately follows.

Recall that  $\okval_{Y_\bullet}(D) \subseteq \oklim_{Y_\bullet}(D)$. Thus
$$
\text{$\okval_{Y_\bullet}(D)=\oklim_{Y_\bullet}(D)$ if and only if  $\vol_{\R^{\kappa(D)}}(\okval_{Y_\bullet}(D))=\vol_{\R^{\kappa(D)}}(\oklim_{Y_\bullet}(D))$.}
$$
Now the second assertion follows from the first assertion.
\end{proof}

\begin{example}
Upon obtaining Theorem \ref{comvallim}, one may wonder whether under the same settings, $\oklim_{Y_\bullet}(D)$ and $\okval_{Y_\bullet}(D)$ coincide up to rescaling by a constant, i.e.,
$$
\oklim_{Y_\bullet}(D)=(\deg(\phi|_{V'} \colon V' \to Z))^{\frac{1}{\kappa(D)}}\cdot\okval_{Y_\bullet}(D).
$$
This is not true in general.
For instance, consider a 3-fold $X:=\P^2 \times \P^1$ with the projections $f \colon X \to \P^2$ and $g \colon X \to \P^1$. Let $H:=f^*L$ and $F:=g^*P$ where $L$ is a line in $\P^2$ and $P$ is a point in $\P^1$. Then $H$ is an abundant divisor with $\kappa(H)=2$.
Note that $f$ is the Iitaka fibration of $H$.
Take a general point $x$ in $X$ and general members $H' \in |H|$ and $S \in |H+2F|$ containing $x$. Note that $S$ is a Nakayama subvariety of $H$ and $\deg(f|_S \colon S \to \P^2)=2$.
We now fix an admissible flag
$$
Y_\bullet : X \supseteq S \supseteq S \cap H' \supseteq \{ x \}
$$
on $X$. It is easy to check that $\okval_{Y_\bullet}(H)$ is an isosceles right triangle in $\{0\} \times \R_{\geq 0}^2$
and $\oklim_{Y_\bullet}(H)$ is a non-isosceles right triangle in $\{0\} \times \R_{\geq 0}^2$.
In particular, we see that $\oklim_{Y_\bullet}(H) \neq \sqrt{2} \cdot  \okval_{Y_\bullet}(H)$.
\end{example}

\begin{example}\label{nonexist}
We see an example of a variety with a pseudoeffective abundant divisor which does not have any Nakayama subvariety $V$ giving rise to a generically injective map $\phi|_{V'} \colon V' \to Z$ (i.e., $\deg \phi|_{V'})=1$) as in Theorem \ref{comvallim}.
Let $S$ be a minimal surface with $\kappa(S)=1$.
Then $K_S$ is semiample, and $\kappa(K_S)=\nu_{\BDPP}(K_S)=1$.
Denote by $\pi \colon S \to C$ the relatively minimal elliptic fibration induced by $|mK_S|$ for $m \gg 0$. Note that $\pi$ is the Iitaka fibration of $K_S$.
Suppose now that $\pi$ has no section. For instance, if $S$ is a Delgachev surface, then $\pi$ has multiple fibers so that $\pi$ has no section. Then, for any Nakayama subvariety $V$ of $K_S$, the map $\pi|_{V} \colon V \to C$ is not generically injective. In particular, by Theorem \ref{comvallim}, $\okval_{Y_\bullet}(K_S)$ and $\oklim_{Y_\bullet}(K_S)$ are different for any admissible flag $Y_\bullet$ on $S$ containing a Nakayama subvariety of $K_S$ such that $Y_2$ is a general point in $S$.
\end{example}


\end{document}